\documentclass[11pt]{amsart}

\usepackage{fullpage}
\usepackage{amsfonts}
\usepackage{amssymb}
\usepackage{ytableau}
\usepackage{tikz}
\usepackage{tikz-cd}
\tikzcdset{scale cd/.style={every label/.append style={scale=#1},
		cells={nodes={scale=#1}}}}
\usepackage{adjustbox}
\usepackage{mathtools}
\usepackage{enumitem}
\usepackage{verbatim}

 \newtheorem{theorem}{Theorem}[section]
 \newtheorem{corollary}[theorem]{Corollary}
 \newtheorem{lemma}[theorem]{Lemma}
 \newtheorem{proposition}[theorem]{Proposition}
 \theoremstyle{definition}
 \newtheorem{definition}[theorem]{Definition}
 \newtheorem{remark}[theorem]{Remark}
 \newtheorem{remarks}[theorem]{Remarks}
  
  \newtheorem{example}[theorem]{Example}
 \newtheorem*{example*}{Example}
  
 \numberwithin{equation}{section}
\DeclareMathOperator{\Ima}{Im}

\DeclareMathOperator{\M}{\mathsf{M}}
\DeclareMathOperator{\GL}{\mathsf{GL}}
\DeclareMathOperator{\m}{\mathsf{mod}}
\DeclareMathOperator{\R}{\mathcal{R}}
\DeclareMathOperator{\F}{\mathcal{F}}
\DeclareMathOperator{\G}{\mathcal{G}}
\DeclareMathOperator{\IS}{\mathcal{IS}}
\DeclareMathOperator{\PT}{\mathcal{PT}}
\DeclareMathOperator{\T}{\mathcal{T}}
\DeclareMathOperator{\Mo}{\mathcal{M}}
\DeclareMathOperator{\I}{\mathbf{I}}
\DeclareMathOperator{\diag}{\mathsf{diag}}
\DeclareMathOperator{\dom}{\mathsf{dom}}
\DeclareMathOperator{\gr}{\mathsf{gr}}
\DeclareMathOperator{\tab}{Tab}
\DeclareMathOperator{\sst}{SST}

\DeclareMathOperator{\csst}{cSST}
\DeclareMathOperator{\Ind}{Ind}
\newcommand{\al}{\alpha}

\newcommand{\la}{\lambda}
\newcommand{\La}{\Lambda}
\newcommand{\lm}{\lambda / \mu}
\begin{document}
\title[Tableaux and orbit harmonics quotients for monoids]{Tableaux and orbit harmonics quotients for finite transformation monoids }

\author[Maliakas]{Mihalis Maliakas}
\address{%
	Department of Mathematics\\
	University of Athens\\
	Greece}
\email{mmaliak@math.uoa.gr}

\author[Stergiopoulou]{Dimitra-Dionysia Stergiopoulou}
\address{%
		Department of Mathematics\\
		University of Thessaly\\
		Greece}
\email{dstergiop@uth.gr}



\date{}

\begin{abstract}We extend Grood’s tableau construction of irreducible representations of the rook monoid and Steinberg’s analogous result for the full transformation monoid. Our approach is characteristic-free  and applies to any submonoid $\mathcal{M}(n)$ of the partial transformation monoid on an $n$-element set that contains the symmetric group. To achieve this,  we introduce and study a functor from the category of rational representations of the monoid of $n \times n$ matrices to  the category of finite dimensional representations of $\mathcal{M}(n)$. We establish two branching rules. Our main results describe graded module structures of orbit harmonics quotients for the rook, partial transformation, and full transformation monoids. This yields analogs of the Cauchy decomposition for polynomial rings in $n\times n$ variables.	
\end{abstract}

\maketitle
\tableofcontents
	\section{Introduction}\label{S1}
The main goal of this paper is to prove characteristic-free filtration  results for the orbit harmonics quotients of three important finite monoids, the rook monoid $\IS_n$, the partial transformation monoid $\PT_n$, and the full transformation monoid $\T_n$. These results are motivated by a theorem of Rhoades \cite[Theorem 4.2]{Rho}  determining the $\mathfrak{S}_n \times \mathfrak{S}_n$-module structure of the orbit harmonics quotient of the symmetric group $\mathfrak{S}_n$ over a field whose characteristic is zero or greater than $n$. To be precise, let $n$ be a positive integer, let $\Bbbk$ be a field whose characteristic is zero or greater than $n$, let $\textbf{x}_{n,n}$ be an $n \times n$ matrix of variables $(x_{i,j})$, and let $\Bbbk[\textbf{x}_{n,n}]$ be the polynomial ring over these variables. Rhoades considers the ideal $I_n$ of $\Bbbk[\textbf{x}_{n,n}]$ generated by all row and column variable
sums and all products of two variables drawn from the same row or column of the matrix $\textbf{x}_{n,n}$. 
He  shows that the quotient $\Bbbk[\textbf{x}_{n,n}]/I_n$ is the orbit harmonics quotient associated to the locus of  $n \times n$ permutation matrices in the set $M_n(\Bbbk)$ of $n \times n$ matrices.  Moreover, he establishes deep connections with the combinatorics of permutations. In particular he shows that the reversed Hilbert series of $\Bbbk[\textbf{x}_{n,n}]/I_n$ is the generating function of permutations in $\mathfrak{S}_n$ by the length of their longest increasing subsequence.

We have the natural action of the group $\mathfrak{S}_n \times \mathfrak{S}_n$ on the matrix $(x_{i,j})$ given by independent row and column permutations. This yields an action on $\Bbbk[\textbf{x}_{n,n}]$ and $I_n$ is an $\mathfrak{S}_n \times \mathfrak{S}_n$-submodule of $\Bbbk[\textbf{x}_{n,n}]$. The ideal $I_n$ is homogeneous and thus $\Bbbk[\textbf{x}_{n,n}]/I_n$ is a graded $\mathfrak{S}_n \times \mathfrak{S}_n$-module. We write $\lambda \vdash n$ to indicate that $\lambda$ is a partition of $n$.  For such a $\lambda$, let $S^\lambda$ be the corresponding Specht module of $\mathfrak{S}_n$. Rhoades showed the following theorem which is very relevant for the present paper.

\begin{theorem}\cite[Theorem 4.2]{Rho}\label{Rho} Suppose $n$ is a positive integer and $\Bbbk$ is a field whose characteristic is zero or greater than $n$. Then for any $r\ge 0$, the degree $r$ piece of $\Bbbk[\textbf{x}_{n,n}]/I_n$  has $\mathfrak{S}_n \times \mathfrak{S}_n$-module structure  \[\big( \Bbbk[\textbf{x}_{n,n}]/I_n \big)_r \cong \bigoplus_{\substack{\lambda \vdash n\\ \lambda_1 = n-r}} S^\lambda \otimes S^\lambda,\] where the direct sum is over all partitions $\lambda$ of $n$ with first part $\lambda_1=n-r$.
\end{theorem}
Hence this result yields the irreducible decomposition of $\big( \Bbbk[\textbf{x}_{n,n}]/I_n \big)_r$ as an $\mathfrak{S}_n \times \mathfrak{S}_n$-module if the characteristic is zero or greater than $n$.
 
\subsection*{Outline of the paper} This paper concerns submonoids $\Mo(n)$ of $\PT_n$ that contain $\mathfrak{S}_n$ with particular emphasis on $\IS_n, \PT_n$ and $\T_n$. Our main contributions are summarized in items (i)-(iv) below. Let $\Bbbk$ be an infinite field of arbitrary characteristic. By   $P_{\Bbbk}(n,r)$ we denote the category of finite dimensional rational representations of degree $r$ of the monoid $\M_n(\Bbbk)$ of $n \times n$ matrices over $\Bbbk$.

\textbf{(i)} For any monoid $\Mo(n)$ as above, we define a functor  $\G_{\Mo (n)}: P_{\Bbbk}(n,r) \to \m \Mo(n),$ the symmetrized Schur functor (see Definition \ref{mSchurf}). We examine how this is related to the classical Schur functor  (see Corollary \ref{cormain1}). We also show that $\G_{\Mo(n)}(M)$ is an induced module from the symmetric group $\mathfrak{S}_r$ (see Theorem \ref{ind}).

\textbf{(ii)} Utilizing the symmetrized Schur functor, we define  $\Mo(n)$-modules $\R(n)^{\la / \mu}$ for any submonoid $\Mo(n)$ of $\PT_n$ and any skew partition $\la / \mu$ (see Definition \ref{Specht}). These modules were motivated by the constructions of Grood \cite{Gro} and Steinberg \cite[Section 5.3]{Ste} of the irreducible representations of the rook monoid and the full transformation monoid, respectively, over the complex numbers.  Our construction is characteristic-free and uniform. From the basis theorem for skew Schur modules of $\GL_n(\Bbbk)$ \cite{ABW}, we obtain an analogous result for the modules $\R(n)^{\la / \mu}$ (see Theorem \ref{bm}). We prove two branching rules for the $\Mo(n)$-modules $\R(n)^{\la / \mu}$.  These generalize Solomon’s branching rule for the rook monoid $\IS_n$ over the complex numbers \cite[Corollary 3.15]{Sol} and the branching rule for Specht modules of James and Peel \cite[3.1 Theorem]{JaPe} (see Theorem \ref{br2}).

\textbf{(iii)}  Let $\textbf{x}_{m,n}$ be an $m \times n$ matrix of variables. If $\lambda$ is a partition of a nonnegative integer $r$, we write $\la \vdash r$. 
The group $\GL_m({\Bbbk}) \times \GL_n(\Bbbk)$ acts on the polynomial ring $\Bbbk[\textbf{x}_{m \times n}]$ in a natural way and when the characteristic of $\Bbbk$ is zero, the irreducible decomposition of the degree $r$ homogeneous component $\Bbbk[\textbf{x}_{m \times n}]_r$ is $\bigoplus_{\la \vdash r} L_\la(V_m) \otimes L_\la(V_n)$, see \cite[p. 121]{F}. Here,  $V_m$ is the natural $\GL_m({\Bbbk})$-module and  $L_{\lambda}(V_m)$ is the corresponding Schur module associated to the partition $\lambda$. In arbitrary characteristic, the previous decomposition holds up to filtration \cite[(3.2.5) Theorem]{W}. This is referred to as the Cauchy decomposition of $\Bbbk[\textbf{x}_{m \times n}]$ since the corresponding identity of formal characters is known as the Cauchy identity \cite[7.12.1 Theorem]{St}. Related results appear in various contexts, for example invariant theory \cite{DKR, DEP, Ma}, highest weight categories \cite{Kr} and cellular algebras \cite{Ax}. In Section 5 (Definition \ref{Iideals}) we consider natural homogeneous ideals $J_{m,n}(Z)$ of $\Bbbk[\textbf{x}_{m \times n}]$ associated to the monoids $\IS_n$, $\PT_n$, $\T_n$. These ideals are closely related to Rhoades's ideal $I_n$. When $m=n$, the ideal $J_{n,n}(\IS)$, for example, is generated by the products of any two variables lying in the same row or in the same column, while $I_n$ is obtained from $J_{n,n}(\IS)$ by adjoining all row sums and all column sums.  We prove the following result on the graded structure of the quotient rings (see Theorem \ref{m2}), which we regard as a characteristic-free analog  for the monoids $\IS_n$, $\PT_n$ and $\T_n$ of Theorem \ref{Rho} mentioned above. For a partition $\lambda$ we denote the transpose partition by $\lambda'$.
\begin{theorem}\label{Im2} Suppose $m,n,r$ are positive integers.
\begin{enumerate}[leftmargin=*]
	\item[\textup{(1)}] The $\IS_m \times \IS_n$-module $\big(\Bbbk[\textbf{x}_{m \times n}] / J_{m,n}(\IS)\big)_r$ has a filtration with quotients \[{\R(m)}^{\lambda} \otimes {\R(n)}^\lambda, \ \lambda \vdash r, \] each appearing exactly once.
	\item[\textup{(2)}] The $\GL_m(\Bbbk) \times \PT_n $-module $\big(\Bbbk[\textbf{x}_{m \times n}] / J_{m,n}(\PT)\big)_r$ has a filtration with quotients \[L_{\lambda'}(V_m) \otimes {\R(n)}^\lambda, \ \lambda \vdash r,\] each appearing exactly once.
	\item[\textup{(3)}] The $\mathfrak{S}_m \times \T_n $-module, $\big(\Bbbk[\textbf{x}_{m \times n}] / J_{m,n}(\T)\big)_r$ has a filtration with quotients \[L_{\lambda'}(U_m) \otimes {\R(n)}^\lambda, \ \lambda \vdash r,\] each appearing exactly once.
	\end{enumerate}\end{theorem} Thus we obtain analogs of the Cauchy decomposition of the polynomial ring $\Bbbk[\textbf{x}_{m \times n}]$ and the basis theorem of D\'esarm\'enien,  Kung and Rota \cite{DKR}. 

\textbf{(iv)} We continue with the setup of (iii) and assume in addition that $m=n$. Consider the monoids $\IS_n, \PT_n$ and $\T_n$ as affine varieties in $\M_{n}(\Bbbk)$.  Orbit harmonics is a general method that associates to a finite locus \(Z \subseteq \Bbbk^N\) in affine
space a graded quotient ring obtained from the associated graded ideal
\(\gr(\I(Z))\) of its vanishing ideal \(\I(Z)\). Thus, although the coordinate
ring $\Bbbk[\textbf{x}_N]/\I(Z)$ of $Z$ is usually not graded, the
quotient \(\Bbbk[\textbf{x}_N]/\gr(\I(Z))\) is graded and has the same
underlying vector space dimension. This method has a long history, going
back at least to Kostant's work on Lie group representations on polynomial
rings \cite{Kos}, and appears prominently in the work of Garsia and Procesi
on graded \(\mathfrak S_n\)-modules and \(q\)-Kostka polynomials
\cite{GaPro}. Rhoades applied orbit harmonics to the locus of permutation matrices in $M_n(\Bbbk)$ \cite{Rho}. Using this method, we prove that the above ideals $J_{n,n}(Z)$ are the associated graded ideals $\gr(\I(Z))$ of the corresponding vanishing ideals, where $Z$ is any of $\IS_n, \PT_n, \T_n$. Thus from Theorem \ref{Im2} we obtain results on the graded module structure of the orbit harmonics quotients $\Bbbk[\textbf{x}_{n \times n}]/ \gr(\I(Z))$ (see Corollary \ref{m3}). We regard this and Theorem \ref{Im2} as the main results of the paper.

Throughout the paper we take the point of view of utilizing the polynomial representation theory of the general linear group $\GL_n(\Bbbk)$ \cite{ABW,Gr,W} and related combinatorics. This is done using rational modules of the monoid $\M_n(\Bbbk)$ and the symmetrized Schur functor. Thus our methods differ significantly from \cite{Rho}. The lack of semisimplicity in positive characteristic (and even in characteristic zero for the monoids $\PT_n$ and $\T_n$) necessitates the use of suitable filtrations in place of direct sum decompositions. However, we are not aware of an a priori reason that guarantees the existence of such filtrations in arbitrary characteristic.  Alternatively, one could use induced modules from the symmetric group based on work of Clifford \cite{Cli}, Munn \cite{Mun1, Mun2} and Ponizovskii \cite{Poz} and its modern recast in terms of Green's idempotent theory \cite{Gr}  by Ganyushkin, Mazorchuk and Steinberg \cite{GaMaSt}.  However, our approach seems very suitable for the purposes of Theorem \ref{Im2} and Corollary \ref{m3}. For an approach to the representations of the rook monoid $\IS_n$ using Schur algebras, we refer to Andre and Martins \cite{AnMar}. 

Section 2 is devoted to preliminaries concerning polynomial representations of $\GL_n(\Bbbk)$ and representations of $\mathfrak{S}_n$ that will be used in the sequel. In Section 3 we define the symmetrized Schur functor and study its basic properties. In Section 4 we introduce analogs $\R(n)^{\la / \mu}$ and $\R(n)_{\la / \mu}$ of skew Specht modules $S^{\la / \mu}$ and dual skew Specht modules $S_{\la / \mu}$. We establish for these a basis theorem and two branching rules. As an application of the first branching rule we show that when the characteristic of $\Bbbk$ is zero, the $\Mo(n)$-modules $\R^\lambda$ are distinct in most cases as $\lambda$ varies over the partitions of $r$ with $r \le n$. Section 5 contains Theorem \ref{Im2} and its corollary  on orbit harmonics quotients. In Appendix A we prove a skew symmetric version of Theorem \ref{Im2}.

\section{Preliminaries}
In this section, we establish notation and  gather preliminaries that will be used in the sequel. Throughout this paper, $\Bbbk$ is an infinite field of arbitrary characteristic unless specified otherwise.

\subsection{Monoids}\label{pm} Let $\PT_n$ be the subset  of $\M_n(\Bbbk)$ of matrices with entries from $\{0,1\}$ with at most one entry equal to 1 in each column. Under matrix multiplication, $\PT_n$ is a monoid known as the \textit{partial transformation monoid}. The submonoid $\IS_n$ of $\PT_n$ consisting of the matrices with at most one entry equal to 1 in every column and every row is called the \textit{symmetric inverse monoid}. This is also known as the \textit{rook monoid} because of the correspondence between matrices in $\IS_n$ and placements of non-attacking rooks on an $n \times n$ chessboard. The submonoid $\T_n$ of $\PT_n$ consisting of matrices that have exactly one entry equal to 1 in each column is called the \textit{full transformation monoid}. The symmetric group $\mathfrak{S}_n$ may be realized as the group of permutation matrices in $\M_n(\Bbbk)$.
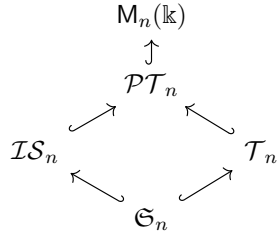
\begin{center}
	\begin{tikzcd}[scale cd=0.9,sep=small]
		& & \M_n(\Bbbk) & & \\
		& & \PT_n \arrow[hookrightarrow, u] & & \\
		& \IS_n \arrow[hook,ru] & & \T_n \arrow[hook',lu] &\\
		& & \mathfrak{S}_n \arrow[hook',lu] \arrow[hook,ru]& & 
	\end{tikzcd}
\end{center}
We record the cardinalities.
\begin{proposition}\label{card}	The following hold.
	\[|\IS_n| = \sum_{r=0}^{n} \tbinom{n}{r}^2r!, \ \ |\PT_n| = (n+1)^n, \ \ |\T_n| = n^n.\]
\end{proposition}

For a positive integer $n$, we let $[n]:=\{1,2,\dots, n\}$. A \textit{partial map} $p$ of $[n]$ is a map from a subset $\dom(p)$ of $[n]$, called the \textit{domain} of $p$, to $[n]$. If
$p_1 : [n] \to [n]$ and $p_2 : [n] \to [n]$ are partial maps, we have the partial map given by their composition $p_2 p_1 : [n] \to [n]$, $p_2 p_1 (a)= p_2(p_1(a))$ for all $a \in p_1^{-1}(\dom(p_2))$. For $A=(a_{ij}) \in \PT_n$, we define the corresponding partial map $p_A:[n] \to [n]$ by 
$
\dom(p_A)=\{j\in[n]: a_{ij}=1 \text{ for some } i\in[n]\}$ and 
$p_A(j)=i \text{ if } a_{ij}=1.$
 Conversely, if $p:[n]\to [n]$ is a partial map, we have the matrix $A_p=(a_{ij}) \in \PT_n$ defined by $a_{ij}=1$ if and only if $j\in \dom(p)$  and  $p(j)=i$. We see that the set of partial maps $[n] \to [n]$ is a monoid with respect to composition and is isomorphic to $\PT_n$.
	
	\subsection{Skew partitions and tableaux}\label{s2.1}
	A \textit{partition} of  an  integer $r$ is  a  finite sequence  of  weakly  decreasing  nonnegative integers that sum to $r$. For a partition $\lambda$, we denote by $\ell(\lambda)$ the number of nonzero terms of the sequence. We identify two partitions if they differ by a string of zeros. For partitions $\la=(\la_1, \dots, \la_q)$ and $\mu=(\mu_1, \dots, \mu_q)$ we write $\mu \subseteq \la$ if $\mu_i \le \la_i$ for all $i$. A \textit{skew partition} $\la / \mu$ is a pair of partitions $\la $ and $\mu$ such that $\mu \subseteq \la$. By a slight abuse of notation, we usually write $\la / \mu =(\la_1 - \mu_1, \dots, \la_q - \mu_q)$. If $\mu=(0)$, we identify the skew partition $\la / (0)$ with the partition $\lambda$. For a skew partition $\la / \mu$ we write $|\la / \mu|=r$ if $|\la|-|\mu|=r$ and we say that $\la / \mu$ is a skew partition of $r$.

	The \textit{diagram} of a partition  $\lambda = (\lambda_1, \dots, \lambda_q)$ is a collection of cells arranged in left-justified rows with a decreasing number of cells in each row from top to bottom. The diagram of a skew partition $\la / \mu$ may be obtained from the diagram of $\lambda$ by omitting the cells corresponding to $\mu$. For example, if $\lambda=(4,3,3)$ and $\mu=(2,1)$, then the diagram of $\la / \mu$ is  \begin{center}
		\ytableausetup{smalltableaux}\ydiagram{2+2,1+2,0+3}
	\end{center}
	
	For a skew partition $\la / \mu$, a \textit{tableau} of shape $\la / \mu$ with entries from $[n]$ is a filling of the diagram of $\la / \mu$ with elements of $[n]$. Let $\tab_{\lambda / \mu}([n])$ be the set of such tableaux. The \textit{weight} $\al=(\al_1,\dots, \al_n)$ of a tableau $T \in \tab_{\lambda / \mu}([n])$ is defined by letting $\al_i$ be the number of appearances of the entry $i$ in $T$. We also say that the weight of $i$ in $T$ is equal to $\al_i$. A  tableau $T \in \tab_{\lambda / \mu}([n])$ is called \textit{semistandard} (respectively \textit{co-semistandard}) if the entries are strictly (respectively weakly) increasing across the rows from left to right and  weakly (respectively strictly) increasing in the columns from top to bottom. (In \cite{ABW}, the term `standard' is used for our semistandard, and `co-standard' for our co-semistandard.) The subset of $\tab_{\lambda / \mu}([n])$ of semistandard (respectively co-semistandard) tableaux of shape $\la / \mu$ will be denoted by $\mathrm{SST}_{\la / \mu}([n])$ (respectively $\mathrm{cSST}_{\la / \mu}([n])$). 
	 
	\subsection{Schur modules and Weyl modules}\label{Sm}
	Suppose $V$ is a finite dimensional vector space. Let $\La(V)=\bigoplus_{i \ge 0}\La^i(V)$ be the exterior algebra of $V$ with the usual grading and let $Sym(V) = \bigoplus_{i\ge 0}Sym_i(V)$ be the symmetric algebra of $V$ with usual grading. We also have the divided power algebra $D(V)=\bigoplus_{i\geq 0}D_i(V)$ (\cite[Section I.4]{ABW}) of $V$, where $D_i(V)$ is defined as the dual of $Sym_i(V^*)$ and $V^*$ denotes the dual of $V$.   If $\al=(\al_1,\dots, \al_q)$ is a sequence of nonnegative integers, we use the notation $\La^{\al}(V):=\La^{\al_1}(V) \otimes \dots \otimes \La^{\al_q}(V)$  and likewise $Sym_{\al}(V):=Sym_{\al_1}(V) \otimes \dots \otimes Sym_{\al_q}(V)$ and  $D_{\al}(V):=D_{\al_1}(V) \otimes \dots \otimes D_{\al_q}(V)$.
	
	Suppose $\la / \mu$ is a skew partition. The \textit{Schur module} $L_{\lambda / \mu}(V)$ and the \textit{Weyl module} $K_{\lm}(V)$ are defined in \cite[Definition II.1.3]{ABW} as the images of particular $GL(V)$-maps\begin{align*} d_{\lm} (V):\La^{\lambda / \mu}(V) \to Sym_{\lambda' / \mu'}(V), \  \ d'_{\lm}(V): D_{\lm}(V) \to \La^{\lambda' / \mu'} (V).\end{align*} For example,  if $\lambda=(r)$, then $L_{\lambda}(V)=\La^r(V)$ and $K_{\lambda}(V)=D_r(V)$. If $\lambda=(1^r)$, then $L_{\lambda}(V)=Sym_r(V)$ and $K_{\lambda}(V)=\La^r(V)$.
	
	Suppose $\{e_1, \dots, e_n\}$ is a basis of the vector space $V$ and  $T \in \mathrm{Tab}_{\la / \mu}([n])$ is a  tableau. We write $T(i,j)$ for the entry of $T$ in position $(i,j)$. Define elements $X_T \in \La^{\lm}(V)$ and $Y_T \in D_{\la / \mu}(V)$ by \begin{align*}
		X_T:=e_{T(1,1)} \cdots e_{T(1,\lambda_1 - \mu_1)}\otimes \cdots\otimes e_{T(q,1)} \dots e_{T(q,\lambda_q - \mu_q)}.
	\end{align*}
	where $\la = (\la_1, \dots, \la_q)$ and $\mu = (\mu_1, \dots, \mu_q)$, and
		\begin{align*}
		Y_T:=e_1^{(a_{11})} \cdots e_n^{(a_{1n})}  \otimes \cdots \otimes  e_1^{(a_{q1})} \cdots e_n^{(a_{qn})},
	\end{align*}
	where $a_{ij}$ is equal to the number of entries in row $i$ of $T$ that are equal to $j$. 
	\begin{theorem}\cite[Theorem II.2.16, Theorem II.3.16]{ABW}\label{Bthm} Let $V$ be a finite dimensional vector space, $\{e_1, \dots, e_n\}$ a basis of $V$ and $\la / \mu$ a skew partition. Then a basis \begin{itemize}[leftmargin=*] \item of the Schur module $L_{\lm}(V)$ is the set $\{ d_{\lm}(V)(X_T): T \in \sst_{\lm}([n])\}$, \item of the Weyl module $K_{\lm}(V)$ is the set $\{ d'_{\lm}(V)(Y_T): T \in \csst_{\lm}([n])\}.$  \end{itemize}
	\end{theorem}
	
	We will need the following proposition. The proof of this follows from the proof of Theorem \ref{Bthm} given in \cite{ABW}.
	\begin{proposition}[Straightening law]\label{strl}Suppose $T \in \tab_{\lambda / \mu}([n])$.\begin{itemize}[leftmargin=*]
			\item There exist semistandard tableaux $T_i \in \sst_{\lm}([n])$ such that the weight of each $T_i$ is equal to the weight of $T$ and $d_{\lm}(V)(X_T)$ is a linear combination of the $d_{\lm}(V)(X_{T_i})$.
				\item There exist co-semistandard tableaux $T_i \in c\sst_{\lm}([n])$ such that the weight of each $T_i$ is equal to the weight of $T$ and $d'_{\lm}(V)(Y_T)$ is a linear combination of the $d'_{\lm}(V)(Y_{T_i})$.
		\end{itemize}\end{proposition}
\subsection{Polynomial representations of $\GL_n(\Bbbk)$ and rational representations of $\M_n(\Bbbk)$}\label{eqCat}  By $M_{\Bbbk}(n,r)$ we denote the category of homogeneous polynomial representations of $\GL_n(\Bbbk)$ of degree $r$ and by $P_{\Bbbk}(n,r)$ we denote the category of rational representations of $\M_n(\Bbbk)$ of degree $r$ \cite[p. 4-5]{Gr}. Every module $M$ in $P_{\Bbbk}(n,r)$ may be considered a module in $M_{\Bbbk}(n,r)$ via  the inclusion $\GL_n(\Bbbk) \hookrightarrow \M_n(\Bbbk)$. Now since $\Bbbk$ is an infinite field, the set $\GL_n(\Bbbk)$ is Zariski dense in $\M_n(\Bbbk)$. This implies that for every module $M$ in $M_{\Bbbk}(n,r)$, the action of $\GL_n(\Bbbk)$ may be extended uniquely to an action of $\M_n(\Bbbk)$ on $M$ yielding an object in $P_{\Bbbk}(n,r)$. Thus we may regard homogeneous polynomial representations of $\GL_n(\Bbbk)$ of degree $r$ as rational representations of $\M_n(\Bbbk)$ of degree $r$ and conversely. In fact we have the following, see \cite[Example 3, \textsection 1]{Gr} or \cite[2.2 Proposition]{Dot}. \begin{proposition}\label{eqcat}The categories $M_{\Bbbk}(n,r)$ and $P_{\Bbbk}(n,r)$ are equivalent. \end{proposition} In this paper, we will often use concepts and results for the category $M_{\Bbbk}(n,r)$  (from \cite{ABW}, \cite{Gr} and \cite{W}) and state them in their $P_{\Bbbk}(n,r)$ equivalent form. For example, we may regard the Schur functor $f: M_{\Bbbk}(n,r) \to \m \mathfrak{S}_r$ (see the next section) as a functor from $P_{\Bbbk}(n,r)$ to $\m \mathfrak{S}_r$. 

\subsection{Schur functor and Specht modules}\label{Sch} Let $\La(n,r)$ be the set of sequences $\al=(\al_1, \dots, \al_n)$ of nonnegative integers such that $\al_1+\cdots+\al_n=r$. We call the elements of $\La(n,r)$ \textit{weights}. We denote the subset of diagonal matrices in $\GL_n(\Bbbk)$ by $T_n(\Bbbk)$. For a module $M \in M_{\Bbbk}(n,r)$ and a weight $\al \in \La(n,r)$, define the $\al$- \textit{weight space} of $M$ by \[M^{\al}:= \{ v \in M: \diag (t_1,\dots,t_n)v = t^{\al_1}_1\cdots t^{\al_n}_nv, \mathrm{\ for \ all} \diag (t_1, \dots, t_n) \in T_n(\Bbbk) \}.\] Since $\Bbbk$ is an infinite field, we have the weight space decomposition of $M$ \begin{equation}\label{ws}M=\bigoplus_{\al \in \La(n,r)}M^{\al}.\end{equation}
If $h:M \to N$ is a morphism in $P_{\Bbbk}(n,r)$, then for every $\al \in \La(n,r)$ \begin{equation}\label{wpres} h(M^\al) \subseteq N^\al.\end{equation}

We have the action of $\mathfrak{S}_n$ on  $\Lambda(n,r)$ given by  \begin{equation}\label{sigmaa}\sigma\alpha = (\alpha_{\sigma^{-1}(1)}, \dots, \alpha_{\sigma^{-1}(n)}). \end{equation}

	Suppose $n \ge r$. Recall from \cite{Gr} that the \textit{Schur functor} is a functor $f$ from the category $M_{\Bbbk}(n,r)$ of homogeneous polynomial representations of  $\GL_n(\Bbbk)$ of degree $r$ to the category of $\mathfrak{S}_r$-modules. For $M$ an object in the first category,  $f(M)$ is the weight subspace $M^\alpha$ of $M$, where $\alpha = (1^r, 0^{n-r})$ and for $\theta :M \to N$ a morphism in the first category, $f(\theta)$  is the restriction $M^\alpha \to N^\alpha$ of $\theta$.
	
	For a skew partition $\la / \mu$ let us define the \textit{skew Specht module} $S^{\lambda / \mu }$ and the \textit{dual skew Specht module} $S_{\lambda / \mu }$ by \begin{equation}\label{fdef}S^{\lambda / \mu }:=f(L_{\lambda' / \mu'}(V_n)), \ \ S_{\lambda / \mu }:=f(K_{{\la} / \mu}(V_n)),\end{equation} where $V_n$ is the natural $\GL_n(\Bbbk)$-module of column vectors. These are  $\mathfrak{S}_r$-modules for $r= |\la / \mu|$. For example, $S^{(r)}$ is the trivial representation and $S^{(1^r)}$ is the sign representation.
	
	Let $f^{\la / \mu} := \dim(S^{\la / \mu}) \ \text{and} \ f_{\la / \mu} := \dim(S_{\la / \mu}).$ 
	Recall that a \textit{standard tableau} $T \in \tab_{\lambda / \mu}([r])$ is a semistandard tableau that has distinct entries. Then both $f^{\lambda / \mu}$ and $f_{\lambda / \mu}$ are equal to the number of standard tableaux $T \in \sst_{\lambda' / \mu'}([r])$. 

	\subsection{Table of frequently used notation} For the reader's convenience we gather here notation used frequently in the paper.
	\begin{itemize}[leftmargin=*]\smaller{
		\item $\IS_n, \PT_n$ and $\T_n$: symmetric inverse monoid, partial transformation monoid and full transformation monoid, respectively. Section \ref{S1}.
		\item $\M_n(\Bbbk)$, $\GL_n(\Bbbk)$ and $\mathfrak{S}_n$: monoid of $n\times n$ matrices over $\Bbbk$, general linear group of $n\times n$ matrices over $\Bbbk$ and symmetric group of degree $n$, respectively.
		\item $V_n$ and $U_n$: natural $\GL_n(\Bbbk)$-module and standard $\mathfrak{S}_n$-module, respectively. Section \ref{Sch} and Section \ref{LS}.
		\item $L_{\lambda / \mu}(V_n)$, $K_{\lambda / \mu}(V_n)$, $S^{\lambda / \mu}$ and $S_{\lambda / \mu}$: Schur module, Weyl module, Specht module and dual Specht module corresponding to the skew partition $\lambda / \mu$, respectively. Section \ref{Sm} and Section \ref{Sch}.
		\item $\Mo(n)$: submonoid of $\PT_n$.
		\item $P_{\Bbbk}(n,r)$ and $M_{\Bbbk}(n,r)$: category of rational (respectively, polynomial) representations of $\M_n(\Bbbk)$ (respectively, $\GL_n(\Bbbk)$) of degree $r$. Section \ref{eqCat}.
		\item $\mathcal{F}: P_{\Bbbk}(n,r) \to \m \mathfrak{S}_n$: the functor that assigns to every  $P_{\Bbbk}(n,r)$-module $M$ the $\mathfrak{S}_n$-module $M'$. Section \ref{Ff}.
		\item $\G_{\Mo (n)}: P_{\Bbbk}(n,r) \to \m \Mo(n)$: symmetrized Schur functor associated to $\Mo(n)$. Section \ref{Gf}.
		\item $\Lambda(n,r)'$ and $\Lambda(n,r)''$: particular sets of weights. Definition \ref{X}.
		\item $M'$ and $M''$: particular subspaces of a $P_{\Bbbk}(n,r)$-module $M$. Definition \ref{X}.
		\item ${\R(n)}^{\lambda / \mu}$ and ${\R(n)}_{\lambda / \mu}$: analogs of Specht and dual Specht modules for $\Mo(n)$ corresponding to the skew partition ${\lambda / \mu}$, respectively. Definition \ref{Specht}.
		\item $J_{m,n}(\PT)$, $J_{m,n}(\IS)$ and $J_{m,n}(\T)$: ideals of $\Bbbk[\textbf{x}_{m \times n}]$ associated to the monoids $\PT_n, \IS_n$ and $\T_n$, respectively. Definition \ref{Iideals}.
	}\normalsize
	\end{itemize}

	\section{Symmetrized Schur functor}
In Section 3.2 we define for any submonoid $\Mo(n)$ of the partial transformation monoid $\PT_n$ a functor	$\G_{\Mo(n)}:P_{\Bbbk}(n,r) \to \m\Mo(n)$ that attaches to a rational representation of the monoid $\M_n(\Bbbk)$ a finite dimensional representation of $\Mo(n)$. First we need some preparation concerning weights of $P_{\Bbbk}(n,r)$-modules. This is done in Section 3.1.  In Section 3.4 we study properties of the functor	$\G_{\Mo(n)}$.

\subsection{Weight spaces and submonoids of $\PT_n$} We define here certain weights of a $P_{\Bbbk}(n,r)$-module $M$ that are central for our purpose and we show a relevant lemma.
\begin{definition}\label{X} Suppose $n,r$ are positive integers and $M$ is a $P_{\Bbbk}(n,r)$-module. \begin{itemize}[leftmargin=*]\item Let $\Lambda(n,r)'$ be the subset of $\Lambda(n,r)$ consisting of weights $\alpha =(\alpha_1, \dots, \alpha_n)$ such that  $\alpha_i \in \{{0,1}\}$ for all $i$. Define $M':=\bigoplus_{\alpha \in \Lambda(n,r)'} M^\alpha$.\item Let $\Lambda(n,r)''$ be the subset of $\Lambda(n,r)$ consisting of weights $\alpha =(\alpha_1, \dots, \alpha_n)$ such that $\alpha_j \ge 2$ for some $j$. Define $M'':=\bigoplus_{\alpha \in \Lambda(n,r)''} M^\alpha$. \end{itemize}
\end{definition}

The set $\Lambda(n,r)''$ is the complement of $\Lambda(n,r)'$ in $\Lambda(n,r)$ and thus by (\ref{ws}) \begin{equation}\label{sum}M = M' \bigoplus M''.\end{equation}  

 \begin{lemma}\label{sumSn} Let $M \in P_{\Bbbk}(n,r)$.\begin{enumerate}[leftmargin=*]\item[\textup{(1)}] $M'$ and $M''$ are $\mathfrak{S}_n$-submodules of $M$.\item[\textup{(2)}] For a submonoid $\Mo(n) \subseteq \PT_n$, $M''$ is an $\Mo(n)$-submodule of $M$ via the inclusion $\Mo(n) \hookrightarrow \M_n(\Bbbk)$.
\end{enumerate}\end{lemma}
\begin{proof}
	We first prove part (2). If $A=(a_{ij}) \in \PT_n$, let $\phi_{A}$ be any map $\phi: [n] \to [n]$ such that for all $i,j \in [n]$ we have 
		$\phi_{A}(j)=i$ if $a_{ij} =1.$
		Such a map $\phi_{A}$ exists since each column of $A$ has at most one entry equal to 1. 
We claim that for all $t_1,...,t_n \in \Bbbk$ we have \[\diag(t_1, \dots, t_n) A = A \diag(t_{\phi_{A}(1)}, \dots, t_{\phi_A(n)}).\]
Indeed, if $A=(a_{ij})$, then from matrix multiplication it follows that the $(i,j)$ entry of $\diag(t_1, \dots, t_n) A$ is equal to $t_ia_{ij}$ and the $(i,j)$ entry of $A \diag(t_{\phi_A(1)}, \dots, t_{\phi_A(n)})$ is equal to $a_{ij}t_{\phi_A(j)}$.  These entries are equal since, if $a_{ij}\neq 0$,  then $\phi_A(j)=i$ by definition of $\phi_{A}$.

Now let $\beta \in \Lambda(n,r)''$ and $v \in M^\beta$. Consider a matrix $A \in \Mo(n)$. We will show that $Av \in M''$.

Choose a map $\phi_A: [n] \to [n]$ as above. Using the above claim, we have for all $t_1, \dots,t_n \in \Bbbk$,
\begin{align}\label{Av}\diag(t_1,t_2, \dots, t_n)Av &= A \diag(t_{\phi_{A}(1)}, \dots, t_{\phi_A(n)})v\\\nonumber&=At_{\phi_A(1)}^{\beta_1}t_{\phi_A(2)}^{\beta_2} \cdots t_{\phi_A(n)}^{\beta_n} v\\\nonumber&
	=t_{\phi_A(1)}^{\beta_1}t_{\phi_A(2)}^{\beta_2} \cdots t_{\phi_A(n)}^{\beta_n} Av. 
\end{align}
The integers $\phi_A(1), \dots, \phi_A(n)$ in the product $t_{\phi_A(1)}^{\beta_1}t_{\phi_A(2)}^{\beta_2} \cdots t_{\phi_A(n)}^{\beta_n}$ may not be distinct. But since for some $j$ we have $\beta_j \ge 2$, it follows that this product is equal to $t_{1}^{\gamma_1}t_{2}^{\gamma_2} \cdots t_{n}^{\gamma_n}$ for some $\gamma=(\gamma_1, \gamma_2, \dots, \gamma_n) \in \Lambda(n,r)$ that satisfies $\gamma_q \ge2$ for some $q$. So $\gamma \in \Lambda(n,r)''$. From eq. (\ref{Av}) it follows that $Av \in M^{\gamma}$. Hence $Av \in M''$ as desired.

To show part (1) of the lemma, let $A \in \M_n(\Bbbk)$ be a permutation matrix and $\sigma:=\phi_A$ the corresponding permutation. Then from (\ref{Av}) we have  $Av \in M^{\sigma\al}$ for every $v \in M^\al$ and $\al \in \La(n,r)$, where $\sigma \al =(\al_{\sigma^{-1}(1)}, \dots, \al_{\sigma^{-1}(n)})$. Thus $M'$ and  $ M''$ are $\mathfrak{S}_n$-submodules of $M$. \end{proof}

\subsection{The symmetrized Schur functor $\G_{\Mo(n)}$ }\label{Gf}
Consider the decomposition (\ref{sum}) for $M \in P_\Bbbk(n,r)$.  If $\eta :M \to N$ is a morphism in the category $P_{\Bbbk}(n,r)$, then the image under $\eta$ of the
		  subspace $M'' $ of $M$ is contained
		   in the subspace $N''$ of
		    $N$ according to (\ref{wpres}). Hence we have the induced map on quotients  ${M}/M'' \to {N}/N''$. By Lemma  \ref{sumSn}(2), this induced map is a homomorphism of $\Mo(n)$-modules when $M$ is an $\Mo(n)$-module, for any submonoid $\Mo(n)$ of $\PT_n$. Thus we may give the following definition.
		
\begin{definition}\label{mSchurf} Let $\Mo(n)$ be a submonoid of $\PT_n$. The \textit{symmetrized  Schur functor} $\G_{\Mo (n)}: P_{\Bbbk}(n,r) \to \m \Mo(n)$ is the functor that assigns   \begin{itemize}\item to a module $M$, the $\Mo(n)$-module $\G_{\Mo(n)}(M):={M}/M'',$
		\item to a morphism $\theta :M \to N$, the induced  morphism $\G_{\Mo(n)}(M) \to \G_{\Mo(n)}(N)$.\end{itemize}
\end{definition}

\begin{example}\label{exSn}Suppose $\Mo(n)$ is the symmetric group $\mathfrak{S}_n$. 
	
	(1) From Lemma \ref{sumSn}(1) we have that $\G_{\mathfrak{S}_n} (M)$ is isomorphic to $M'$ as $\mathfrak{S}_n$-modules for all $M \in P_{\Bbbk}(n,r)$. 
	
	(2) If moreover $r=n$, then the set $ \Lambda(n,r)'$ consists of a single weight $\omega = (1, \dots, 1)$ (the 1 appears $n$ times). Thus in this case $\G_{\mathfrak{S}_n} (M)$ is isomorphic to $f(M)$, where $f: P_{\Bbbk}(n,n)\to \m \mathfrak{S}_n$ is the Schur functor (cf. Section \ref{Sch}).\end{example}

\begin{remarks}Let $\Mo(n)$ be a submonoid of $\PT_n$ and let $M \in P_{\Bbbk}(n,r)$.
	
	(1) Note that the underlying vector space of the $\Mo(n)$-module $\G_{\Mo(n)}(M)$ does not depend on the choice of $\Mo(n)$. The $\Mo(n)$-module  $\G_{\Mo(n)}(M)$ is equal to the restriction of the $\PT_n$-module  $\G_{\PT_n}(M)$ to $\Mo(n)$.

(2) We note that in (\ref{sum}), while  the subspace $M''$ of $M$ is an $\Mo(n)$-submodule of $M$ according to Lemma \ref{sumSn}(2), the other  subspace $M'$ is in general \textit{not} an $\Mo(n)$-submodule of $M$. For example, let $\Mo(n)=\T_n$ be the full transformation monoid and $M=Sym_r(V_n)$ the $r$th symmetric power of the natural $\M_n(\Bbbk)$-module $V_n$, where $2 \le r \le n$. Let $\{e_1, e_2, \dots, e_n\}$ be the canonical basis of $V_n$ and let $A$ be the matrix $A\in \T_n$ such that every element in the first row is equal to $1$. Then the element $e_1e_2\cdots e_r \in Sym_r(V_n)$ has weight $(1^r,0^{n-r})$, while the element $A(e_1e_2\cdots e_r)=e_1^r$ has weight $(r,0^{n-r})$. This is the reason why in Definition  \ref{mSchurf} we have the quotient $\G_{\Mo(n)}(M):=M/M''$ (and not $M'$).
\end{remarks}

\subsection{The functor $\F$}\label{Ff}
We now consider a functor needed in the next subsection.

\begin{definition}\label{F} Let $\mathcal{F}: P_{\Bbbk}(n,r) \to \m\mathfrak{S}_n$ be the functor that assigns   \begin{itemize}\item to a module $M$, the $\mathfrak{S}_n$-module $\mathcal{F}(M):=M',$
		\item to a morphism $M \to N$, the induced  morphism  $M' \to N'.$\end{itemize}
\end{definition}
It follows from Example \ref{exSn}(1) that the functor $\mathcal{F}: P_{\Bbbk}(n,r) \to \m\mathfrak{S}_n$ is naturally equivalent  to the particular symmetrized Schur functor $\G_{\mathfrak{S}_n}: P_{\Bbbk}(n,r) \to \m \mathfrak{S}_n$.
\begin{example} Let $\{e_1, e_2, \dots, e_n \}$ be the canonical basis of the natural $\M_n(\Bbbk)$-module $V_n$. If $M=Sym_r(V_n)$ is the $r$th symmetric power of $V_n$, then $\F(M)$ is the subspace of $Sym_r(V_n)$ spanned by the square free monomials in $e_1, e_2, \dots, e_n$ of degree $r$. If $M=\Lambda^r(V_n)$ is the $r$th exterior power of $V_n$, then $\F(M)=\Lambda^r(V_n)$.
\end{example}

\subsection{Some properties of the symmetrized Schur functor}
\begin{proposition}\label{exactFM}For a submonoid $\Mo(n)$ of $\PT_n$, the functor $ \G_{\Mo (n)}: P_{\Bbbk}(n,r) \to \m\Mo(n)$ is exact. 
\end{proposition}

\begin{proof}By \cite[(3.3b)]{Gr}, if $0 \to M_1 \to M \to M_2 \to 0$ is an exact sequence in $P_{\Bbbk}(n,r)$, then for any weight $\alpha \in \Lambda(n,r)$ the sequence of vector spaces $0 \to M_1^\alpha \to M^\alpha \to M_2^\alpha \to 0$ is exact.  By summing over $\alpha \in \Lambda(n,r)''$ we have  the exact sequence $0 \to \bigoplus_{\alpha}M^{\alpha}_1 \to \bigoplus_{\alpha}M^{\alpha} \to \bigoplus_{\alpha}M^{\alpha}_2 \to 0$. A standard argument using the $3\times3$ lemma yields that $0 \to \G_{\Mo (n)}(M_1)\to \G_{\Mo (n)}(M) \to \G_{\Mo (n)}(M_2) \to 0$ is exact.\end{proof}

\begin{lemma}\label{resFM}
Let $\Mo(n)$ be a submonoid of $\PT_n$ containing $\mathfrak{S}_n$. For a module  $M \in P_{\Bbbk}(n,r)$, the following hold.	\begin{enumerate}[leftmargin=*]
	\item[\textup{(1)}]The restriction of the $\Mo(n)$-module $\G_{\Mo(n)}(M)$ to $\mathfrak{S}_n$ is isomorphic to $\F(M)$.
	\item[\textup{(2)}]$\dim (\G_{\Mo(n)}(M)) = \binom{n}{r} \dim(f(M))$, where $f: P_{\Bbbk}(n,r)\to \m \mathfrak{S}_r$ is the Schur functor.
	\end{enumerate} 
\end{lemma}
\begin{proof}(1) By Definition  \ref{mSchurf} we have $\G_{\Mo(n)}(M)=M/M''$ and by Lemma \ref{sumSn}(1) we have that the restriction of $M/M''$ to $\mathfrak{S}_n \subseteq \Mo(n)$ is isomorphic to $M'$. But $M' =\F(M)$ according to Definition \ref{F}.
	
	(2) As a vector space $\G_{\Mo(n)}(M)$ is isomorphic to $M'=\bigoplus_{\alpha \in \Lambda(n,r)'} M^\alpha$. By \cite[(3.3a Proposition)]{Gr}, we have $M^{\alpha} \cong M^{(1^r,0^{n-r})}$ for every $\alpha \in \Lambda(n,r)'$. By definition of the  Schur functor, we have $M^{(1^r,0^{n-r})}=f(M)$. Since the cardinality of the set $\Lambda(n,r)'$ is equal to $\binom{n}{r}$ the result follows.
\end{proof}

For the proof of Theorem \ref{main1} below, we will need the following lemma on monomials in symmetric powers and divided powers. First some notation. Recall we have the canonical basis $\{e_1, \dots, e_n\}$ of the natural $\M_{n}(\Bbbk)$-module $V_n$. For $\alpha=(\alpha_1, \dots, \alpha_n) \in \Lambda(n,r)$ we denote by $e_\alpha \in Sym_r(V_{n})$ and $e_{(\alpha)} \in D_r(V_{n})$ the monomials \[e_\alpha:=e_{1}^{\alpha_1} \cdots e_{n}^{\alpha_n}, \  e_{(\alpha)}:=e_{1}^{(\alpha_1)} \cdots e_{n}^{(\alpha_n)}.\] We recall that the action of the symmetric group $\mathfrak{S}_n$ on $e_\alpha$ and $e_{(\alpha)}$ is given by \begin{equation}\label{sigmaea}\sigma e_\alpha = {e_{\sigma(1)}}^{\alpha_1} \cdots {e_{\sigma(n)}}^{\alpha_n}, \ \sigma e_{(\alpha)} = {e_{\sigma(1)}}^{(\alpha_1)} \cdots {e_{\sigma(n)}}^{(\alpha_n)}\end{equation} and the action of $\mathfrak{S}_n$ on  $\Lambda(n,r)$ is given by (\ref{sigmaa}).

\begin{remark}\label{hat} Suppose $\alpha = (\alpha_1, \dots, \alpha_n)\in \Lambda(n,r)'$. We define $\widehat{\alpha} \in \Lambda(n,n-r)'$ by \[\widehat{\alpha}= (\widehat{\alpha}_1, \widehat{\alpha}_2, \dots, \widehat{\alpha}_n):=(1-\alpha_1, 1-\alpha_2, \dots, 1-\alpha_n).\] Then for every $\sigma \in \mathfrak{S}_n$ we have $\sigma e_{\widehat{\alpha}} = e_{\widehat{\sigma(\alpha)}}$ and $\sigma e_{(\widehat{\alpha})} = e_{(\widehat{\sigma(\alpha)})}$
	\end{remark}
	\begin{proof} Using eqs. (\ref{sigmaea}) and (\ref{sigmaa}) we have \begin{align*}\sigma e_{\widehat{\alpha}} = {e_{\sigma(1)}}^{\widehat{\alpha}_1}{e_{\sigma(2)}}^{\widehat{\alpha}_2} \cdots {e_{\sigma(n)}}^{\widehat{\alpha}_n},\ \  e_{\widehat{\sigma(\alpha)}}= e_1^{\widehat{\alpha}_{\sigma^{-1}(1)}}e_2^{\widehat{\alpha}_{\sigma^{-1}(2)}} \cdots e_n^{\widehat{\alpha}_{\sigma^{-1}(n)}}.\end{align*}
		The right hand sides of the above equations are equal since multiplication in the symmetric algebra is commutative. The proof of $\sigma e_{(\widehat{\alpha})} = e_{(\widehat{\sigma(\alpha)})}$ is similar.\end{proof}

Suppose we have modules $N_1 \in P_{\Bbbk}(n,r_1)$ and $N_2 \in P_{\Bbbk}(n,r_2)$ with $r_1+r_2 \le n$. The monoid $\M_n(\Bbbk)$ acts on the tensor product $N_1 \otimes N_2$ by $A(u \otimes v)=Au \otimes Av$, where $A \in \M_n(\Bbbk)$ and $u \in N_1, v \in N_2$. Then we have  $N_1 \otimes N_2 \in  P_{\Bbbk}(n,r_1+r_2)$ and we may consider the image $f(N_1\otimes N_2)$ of $N_1\otimes N_2$ under the Schur functor $f:M_{\Bbbk}(n,r_1+r_2) \to \m\mathfrak{S}_{r_1+r_2} $.

The main result of the present section is the next theorem and its corollary.

\begin{theorem}\label{main1}Let  
	$M \in P_{\Bbbk}(n,r)$. Then as $\mathfrak{S}_n$-modules \[\F(M) \cong f(M \otimes Sym_{n-r}(V_{n})) \cong f(M \otimes D_{n-r}(V_{n})),\]
	where $f:P_{\Bbbk}(n,n) \to \m\mathfrak{S}_n $ is the Schur functor. 
\end{theorem}
\begin{proof} In order to show the first isomorphism, recall that $\F(M)=\bigoplus_{\alpha \in \Lambda(n,r)'} M^\alpha$ according to Definition \ref{F}. Let $\alpha=(\alpha_1, \dots, \alpha_n) \in \Lambda(n,r)'$. With the notation of Remark \ref{hat},  let \[\widehat{\alpha}:=(1-\alpha_1, \dots, 1-\alpha_n) \in \Lambda(n,n-r)'\] and consider the element $e_{\widehat{\alpha}} \in Sym_{n-r}(V_n)$,
	where $e_1, e_2, \dots, e_n$ is the canonical basis of the natural module $V_n$ of $\M_n(\Bbbk)$. Now we define a linear map \[\Phi: \F(M) \to f(M \otimes Sym_{n-r}(V_{n})), \ \ \bigoplus_{\alpha \in \Lambda(n,r)'}v_\alpha \mapsto \bigoplus_{\alpha \in \Lambda(n,r)'}v_\alpha \otimes e_{\widehat{\alpha}},\]
	where $v_\alpha \in M^\alpha$. Since $v_\alpha \in M^{\alpha}$ and $e_{\widehat{\alpha}} \in (Sym_{n-r}(V_{n}))^{\widehat{\alpha}}$, we have  \[v_\alpha \otimes e_{\widehat{\alpha}} \in (M \otimes Sym_{n-r}(V_{n}))^{\alpha +\widehat{\alpha}}=(M \otimes Sym_{n-r}(V_{n}))^{(1,...,1)}=f(M\otimes Sym_{n-r}(V_{n})).\] In other words, the image of $\Phi$ is indeed contained in $f(M\otimes Sym_{n-r}(V_{n}))$.
	
	By \cite[(3.3c) Proposition]{Gr} we have \begin{equation}\label{weightdec}(M \otimes Sym_{n-r}(V_{n}))^{(1,...,1)}= \bigoplus_{\alpha \in \Lambda(n,r)'} M^\alpha \otimes (Sym_{n-r}(V_{n}))^{\widehat{\alpha}}.\end{equation} First, we observe that from the right hand side of eq. (\ref{weightdec}) it follows that the map $\Phi$ is onto. Second, we observe that since $\widehat{\alpha} \in \Lambda(n,n-r)'$ for all $\alpha \in \Lambda(n,r)'$, the vector space $(Sym_{n-r}(V_{n}))^{\widehat{\alpha}}$ is 1 dimensional. Thus we conclude from (\ref{weightdec}) and Definition \ref{F} that the vector spaces $\F(M)$ and $(M \otimes Sym_{n-r}(V_{n}))^{(1,...,1)}$ have equal dimensions. Hence $\Phi$ is onto and 1-1.
	
	It remains to be shown that $\Phi$ is a map of $\mathfrak{S}_n$-modules. Let $\alpha \in \Lambda(n,r)'$ and $v_\alpha \in M^\alpha$. On the one hand we have for all $\sigma \in \mathfrak{S}_n$,\begin{equation}\label{eq1}\sigma \Phi(v_{\alpha})=\sigma(v_\alpha \otimes e_{\widehat{\alpha}})=\sigma v_\alpha \otimes \sigma e_{\widehat{\alpha}}.\end{equation} 
By \eqref{Av}, applied to the permutation matrix corresponding to $\sigma$,
we have $\sigma v_\alpha \in M^{\sigma(\alpha)}$. Thus on the other hand we have \begin{equation}\label{eq2} \Phi(\sigma v_{\alpha})=\sigma v_\alpha \otimes e_{\widehat{\sigma(\alpha)}}.\end{equation} From Remark \ref{hat} it follows that the right hand sides of eqs. (\ref{eq1}) and (\ref{eq2}) are equal. Hence $\Phi$ is a map of $\mathfrak{S}_n$-modules.
	
Next, consider the linear map
\[
\Psi: \F(M) \to f(M \otimes D_{n-r}(V_{n})), \qquad
\bigoplus_{\alpha \in \Lambda(n,r)'}v_\alpha
\mapsto
\bigoplus_{\alpha \in \Lambda(n,r)'}v_\alpha \otimes e_{(\widehat{\alpha})},
\]
where $v_\alpha \in M^\alpha$. By the divided-power analogue of
\eqref{weightdec}, we have
\[
(M \otimes D_{n-r}(V_{n}))^{(1,\dots,1)}
=
\bigoplus_{\alpha \in \Lambda(n,r)'}
M^\alpha \otimes (D_{n-r}(V_{n}))^{\widehat{\alpha}}.
\]
Since each space $(D_{n-r}(V_{n}))^{\widehat{\alpha}}$ is one-dimensional,
the same argument as for $\Phi$ shows that $\Psi$ is a vector space
isomorphism. Finally, Remark \ref{hat} gives
\[
\sigma e_{(\widehat{\alpha})}
=
e_{(\widehat{\sigma(\alpha)})},
\]
so the same equivariance argument used for $\Phi$ shows that $\Psi$ is an
isomorphism of $\mathfrak{S}_n$-modules.
 \end{proof}

\begin{corollary}\label{cormain1}
	Let $\Mo(n)$ be a submonoid of $\PT_n$ containing $\mathfrak{S}_n$ and let $M \in P_{\Bbbk}(n,r)$. Then for the restriction $Res^{\Mo(n)}_{\mathfrak{S}_n}(\G_{\Mo(n)}(M))$ of the $\Mo(n)$-module $\G_{\Mo(n)}(M)$ to $\mathfrak{S}_n$ we have \[Res^{\Mo(n)}_{\mathfrak{S}_n}(\G_{\Mo(n)}(M)) \cong \F(M)\cong f(M \otimes Sym_{n-r}(V_{n})) \cong f(M \otimes D_{n-r}(V_{n})).\]
\end{corollary}
\begin{proof} This follows from Theorem \ref{main1} and Lemma \ref{resFM}.\end{proof}

\subsection{Induced modules and $\G_{\Mo(n)}(M)$}In the Clifford-Munn-Ponizovskii theory on the classification of simple modules of a finite monoid, the induction functor from maximal subgroups of the monoid is a central object, see \cite{GaMaSt} or \cite[Chapter 5.2]{Ste}. We show next that $\G_{\Mo(n)}(M)$ is an induced module from $\mathfrak{S}_r$ to $\Mo(n)$. This result will be used in Theorem \ref{irr} below. \begin{definition}\label{Inr}Let $I_{n,r}$ be the set of all injective maps $[r] \to [n]$.\end{definition} Following \cite[p. 62]{Ste} we have that $\Bbbk I_{n,r}$ is a right $\mathfrak{S}_r$-module: for $h \in I_{n,r}$ and $A  \in \mathfrak{S}_r$, define $hA$ to be the composition $[r] \xrightarrow{\sigma} [r] \xrightarrow{h} [n]$ where $\sigma$ is the permutation corresponding to the permutation matrix $A$. One verifies that $\Bbbk I_{n,r}$ is a free $\mathfrak{S}_r$-module with basis the order preserving injective maps $[r] \to [n]$. Hence the rank of $\Bbbk I_{n,r}$ is equal to $\binom{n}{r}$.

We also  have that $\Bbbk I_{n,r}$ is a left $\PT_n$-module: for $h \in I_{n,r}$ and $A \in \PT_n$, define
\begin{equation}
	A h = \begin{cases}
		p_A h, &\text{if} \ p_A h \in I_{n,r},\\
		0, &\text{else},
	\end{cases}
\end{equation}
where $p_A: [n] \to [n]$ denotes the partial map corresponding to $A \in \PT_n$ and $p_A h$ denotes the composition $[r] \xrightarrow{h} [n] \xrightarrow{p_A} [n]$ of the map $h$ and the partial map $p_A$. In fact, $\Bbbk I_{n,r}$ is a $(\PT_n,\mathfrak{S}_r)$-bimodule.
\begin{theorem}\label{ind}
Let $\Mo(n)$ be a submonoid of $\PT_n$ containing $\mathfrak{S}_n$ and let $M \in P_{\Bbbk}(n,r)$.  Then as $\Mo(n)$-modules \[\G_{\Mo (n)}(M) \cong \Bbbk I_{n,r}\otimes_{\mathfrak{S}_r} f(M),\]
where $f:P_{\Bbbk}(n,r) \to \m\mathfrak{S}_r $ is the Schur functor. Thus we have the following commutative diagram of functors, 
\begin{center}
	\begin{tikzcd}
		P_{\Bbbk}(n,r) \arrow[dd,"f" '] \arrow[rd,"\G_{\Mo (n)}"]\\
		& \m \Mo(n) \\
		\m \mathfrak{S}_r \arrow[ru, "\Ind" ']
	\end{tikzcd}
\end{center}
where $\Ind(X) :=\Bbbk I_{n,r}\otimes_{\mathfrak{S}_r} X$ for $X \in \m\mathfrak{S}_r$.
\end{theorem}
\begin{proof} Using that $\Bbbk I_{n,r}$ is a free $\mathfrak{S}_r$-module of rank equal to $\tbinom{n}{r}$ and Lemma \ref{resFM}(2) we have \[\dim (\Bbbk I_{n,r}\otimes_{\mathfrak{S}_r} f(M)) = \tbinom{n}{r} \dim(f(M))=\dim(\G_{\Mo (n)}(M)).\]Thus in order to prove the first claim of the theorem it suffices to show that there is a surjective map of $\Mo(n)$-modules $\Bbbk I_{n,r}\otimes_{\mathfrak{S}_r} f(M) \to \G_{\Mo (n)}(M).$ For $h \in I_{n,r}$ we have the corresponding partial map $[n] \to [n]$ with domain $[r]$ whose restriction to $[r]$ is equal to $h$. Let $A_h \in \PT_n$ be the matrix corresponding to this partial map. Now consider the map $\Phi:\Bbbk I_{n,r}\otimes f(M) \to \G_{\Mo (n)}(M)$ defined by $\Phi(h\otimes m)=A_hm + M''$ for $h \in I_{n,r}$ and $m \in f(M)$. It is straightforward to verify that $\Phi$ is a surjective map of $\Mo(n)$-modules and that it induces a map $\Bbbk I_{n,r}\otimes_{\mathfrak{S}_r} f(M) \to \G_{\Mo (n)}(M)$.

The second claim of the theorem follows from the first.
\end{proof}
 
\section{Basis theorem and branching rules} For a submonoid $\Mo(n)$ of the partial transformation monoid $\PT_n$ we define $\Mo(n)$-modules $\R(n)^{\la / \mu }$ and $\R(n)_{\la / \mu }$. We establish for these a basis theorem and two branching rules.  In Section 4.6 we give an application of the first branching rule when $\Bbbk$ is a field of characteristic zero.

From the present section, we only need Definition \ref{Specht} and Theorem \ref{bm} for Section \ref{mar}.
\subsection{Definition of $\R(n)^{\la / \mu}$ and $\R(n)_{\la / \mu}$}
\begin{definition} \label{Specht}Suppose $\lambda / \mu$ is a skew partition of $r$ and $\Mo(n)$ is a submonoid of $\PT_n$. We define the $\Mo(n)$-modules\begin{align*}
		\R(n)^{\lambda / \mu} :=\G_{\Mo(n)}(L_{\lambda' / \mu'}(V_{n})),\ \ 
 		\R(n)_{\lambda / \mu} :=\G_{\Mo(n)}(K_{\lambda / \mu}(V_{n})).
	\end{align*}
If $\mu=(0)$, we write $\R(n)^{\lambda}$ and $\R(n)_{\lambda}$ for  $\R(n)^{\lambda / \mu}$ and $\R(n)_{\lambda / \mu}$ respectively. Often we will write $\R^{\lambda / \mu}$ and $\R_{\lambda / \mu}$ in place of $\R(n)^{\lambda / \mu}$ and $\R(n)_{\lambda / \mu}$ respectively, if there is no danger of uncertainty regarding $n$.\end{definition}

\begin{example} Suppose $r=n$, the monoid $\Mo(n)$ is the symmetric group $\mathfrak{S}_n$, and $\lambda$ is a partition of $n$. Then from Theorem \ref{main1} we have that, as $\mathfrak{S}_n$-modules,  \[\R(n)^{\lambda}  \cong \F(L_{\lambda'} (V_{n})) \cong f(L_{\lambda'} (V_{n})).\] We know from Section 2.4 that $f(L_{\lambda'} (V_{n}))$ is the Specht module $S^{\lambda}$ and thus $\R(n)^{\lambda} \cong S^{\lambda}$. Likewise, $\R(n)_\lambda$ is isomorphic to the  dual Specht module $S_\lambda$.
\end{example}

\begin{remark}\label{rem52} Suppose the characteristic of $\Bbbk$ is zero. It is well known that the $\M_n(\Bbbk)$-modules $L_{\lambda' / \mu'}(V_n)$ and $K_{\lambda / \mu}(V_n)$ are isomorphic. (This follows, for example, from the fact they have the same formal character.) From this it follows that the $\Mo(n)$-modules $\R(n)^{\lambda / \mu}$ and $\R(n)_{\lambda / \mu}$ are isomorphic for any submonoid $\Mo(n)$ of $\PT_n$.
\end{remark}

\subsection{Basis theorem}
For a partition $\lambda$ of $r$, Grood \cite[Definition 2.1]{Gro} considered semistandard tableaux of shape $\lambda$ that have distinct entries. Similarly we give the following definition. 
\begin{definition}\label{mtab}Let $\lambda/\mu$ be a skew partition of $r$. A tableau $T \in \tab_{\lambda / \mu}$ with entries in $[n]$ is called a $(\lambda / \mu, n,r)$ tableau if the entries of $T$ are distinct. A $(\lambda / \mu, n,r)$ tableau that is semistandard is called a standard $(\lambda / \mu, n,r)$ tableau.
\end{definition}

\begin{remark}\label{numtab} It is clear that the number of standard $(\lambda / \mu, n,r)$ tableaux is equal to $\tbinom{n}{r}f^{\la / \mu}$, since $f^{\la / \mu}$ equals the number of standard tableaux of shape $\la / \mu$ with entries from $[r]$.
\end{remark}
For the next theorem we adopt the notation of Section \ref{Sm}.
\begin{theorem}[Basis theorem]\label{bm}
		Let \(\lambda/\mu\) be a skew partition of \(r\). Then the elements
		\(d_{\lambda'/\mu'}(V_n)(X_T)\), where \(T\) runs over the standard
		\((\lambda'/\mu',n,r)\) tableaux, descend to a basis of
		\(\R(n)^{\lambda/\mu}\). Similarly, the elements
		\(d'_{\lambda/\mu}(V_n)(Y_T)\), where \(T\) runs over the standard
		\((\lambda/\mu,n,r)\) tableaux, descend to a basis of
		\(\R(n)_{\lambda/\mu}\).
		Moreover,
		\[
		\dim \R(n)^{\lambda/\mu}=\binom nr\dim S^{\lambda/\mu},
		\qquad
		\dim \R(n)_{\lambda/\mu}=\binom nr\dim S_{\lambda/\mu}.
		\]\end{theorem} 
\begin{proof}
	The basis elements of Theorem \ref{Bthm}(1) are weight elements of $L_{\lambda' / \mu'} (V_n)$. From this it follows that the elements $d_{\lambda' / \mu'}(V_n)(X_T)$, where $T$ runs over the semistandard $(\lambda' / \mu', n,r)$ tableaux, form a basis of the vector space $\bigoplus_{\alpha \in \Lambda(n,r)'} (L_{\lambda' / \mu'} (V_n))^\alpha=(L_{\lambda' / \mu'} (V_n))'$.
	By Definition \ref{Specht} and eq. (\ref{sum}) we have $\R(n)^{\lambda / \mu} \cong (L_{\lambda' / \mu'} (V_n))'$ as vector spaces. Hence the elements $d_{\lambda' / \mu'}(V_n)(X_T)$, where $T$ runs over the semistandard $(\lambda' / \mu', n,r)$ tableaux, descend to a basis of the quotient vector space $\R(n)^{\lambda / \mu}= L_{\lambda' / \mu'} (V_n) /  (L_{\lambda' / \mu'} (V_n))''$. Now from this and Remark \ref{numtab} the equality $\dim\R(n)^{\lambda / \mu } =\tbinom{n}{r}\dim S^{\lambda / \mu}$ follows.
	
 The proof for $\R(n)_{\lambda / \mu}$ is similar.
\end{proof}
\subsection{Irreducibility in characteristic zero} \begin{theorem}\label{irr}Suppose $\Bbbk$ is a field of characteristic zero and $\Mo(n)$ is a submonoid of $\PT_n$ containing $\mathfrak{S}_n$.
	\begin{enumerate}[leftmargin=*]    \item[\textup{(1)}] If $\IS_n \subseteq \Mo(n)$, then the $\Mo(n)$-module $\R(n)^\la$ is irreducible for every partition $\la$ of $r$, where $r\le n$.
	 \item[\textup{(2)}] If $\T_n \subseteq \Mo(n)$, then the $\Mo(n)$-module $\R(n)^\la$ is irreducible for every partition $\la$ of $r$, where $r \le n$  and $\la \neq (1^{r})$.\end{enumerate}
\end{theorem}
\begin{proof} (1) Suppose $\Mo(n)=\IS_n$. It is straightforward to verify that Grood's proof of \cite[Theorem 4.1]{Gro} carries over to our setting and yields that $\R(n)^\la$ is an irreducible $\IS_n$-module. Suppose $\IS_n \subseteq \Mo(n)$. By the previous case, the restriction of the $\Mo(n)$-module $\R(n)^\la$ to $\IS_n$ is irreducible. Hence $\R(n)^\la$ is an irreducible $\Mo(n)$-module.
	
	(2) By Theorem \ref{ind}, \[\R(n)^\la \cong \Bbbk I_{n,r}\otimes_{\mathfrak{S}_r} f(L_{\lambda'}(V_n))=\Bbbk I_{n,r}\otimes_{\mathfrak{S}_r} S^{\lambda}.\]
	By \cite[Corollary 5.11]{Ste}, $\Bbbk I_{n,r}\otimes_{\mathfrak{S}_r} S^{\lambda}$ is an irreducible $\T_n$-module since $\la \neq (1^{r})$. Thus $\R(n)^\la$ is an irreducible $\T_n$-module. Hence $\R(n)^\la$ is an irreducible $\Mo(n)$-module. 
\end{proof}

\subsection{Filtrations} Let us recall the following definitions. For $N \in P_{\Bbbk}(n,r)$, a  Schur filtration (respectively, Weyl filtration) of $N$ is a sequence of submodules in $P_{\Bbbk}(n,r)$
\begin{equation}\label{Sfil}0=N_s \subseteq N_{s-1} \subseteq \cdots \subseteq N_0=N \end{equation}
such that for $0 \le i \le s-1$ the quotient $N_i / N_{i+1}$ is zero or isomorphic to a Schur module $L_\lambda (V_n)$ (respectively, Weyl module $K_\lambda (V_n)$) for some partition $\la$ of $r$ \cite[Appendix A1]{Do2}. Likewise, for  $Q \in \m \mathfrak{S}_r$, a Specht filtration (respectively, dual Specht filtration) is a  sequence of $\mathfrak{S}_r$-submodules
\[0=Q_t \subseteq Q_{t-1} \subseteq \cdots \subseteq Q_0=Q\]
such that for $0 \le i \le t-1$ the quotient $Q_i / Q_{i+1}$ is zero or isomorphic to a Specht module $S^\lambda $ (respectively, dual Specht module $S_\lambda$) for some partition $\la$ of $r$. 

In the sequel, we will apply the following lemma.

\begin{lemma}\label{basicl}The sequence \begin{equation}\label{ffil}0=f(N_s) \subseteq f(N_{s-1}) \subseteq \cdots \subseteq f(N_0)=f(N) \end{equation} obtained by applying the Schur functor $f:P_{\Bbbk}(n,r) \to \m \mathfrak{S}_r$ to a Schur filtration (respectively, Weyl filtration) of $N$ as in (\ref{Sfil}), is a Specht (respectively, dual Specht) filtration of $f(N)$. If $n \ge r$, then the multiplicity of the Schur module $L_\lambda(V_n)$, where $\lambda \vdash r$, as a factor of the filtration (\ref{Sfil}) is equal to the multiplicity of the Specht module $S^{\lambda'}$ as a factor of the filtration (\ref{ffil}).\end{lemma} This follows since the functor $f$ is exact (see \cite[(3.3b) Proposition]{Gr}) and satisfies eqs. (\ref{fdef}). 
\subsection{First branching rule}We consider the next theorem as the main result of this section. 

A skew partition such that every column in the corresponding diagram has at most one cell is called a \textit{horizontal strip}.
\begin{theorem}[Branching rule for $\mathfrak{S}_n \subseteq \Mo(n)$]\label{br1} Let $\Mo(n)$ be a submonoid of $\PT_n$ containing $\mathfrak{S}_n$. Let $\lambda$ be a partition of $r$, where $r \le n$, and let $\mathrm{HS}(\lambda, n)$ be the set of all partitions $\lambda^{+}$ of $n$ containing $\lambda$ such that the skew shape $\lambda^{+} / \lambda $ is a horizontal strip. Then the following hold.
	\begin{enumerate}[leftmargin=*]
		\item[\textup{(1)}] The $\mathfrak{S}_n$-module $\R(n)^\lambda$ has a Specht filtration with factors $S^{\lambda^+}$ each appearing exactly once, where ${\lambda^+}$ ranges over $\mathrm{HS}(\lambda, n)$.
		\item[\textup{(2)}] The $\mathfrak{S}_n$-module $\R(n)_\lambda$ has a dual Specht filtration with factors $S_{\lambda^+}$ each appearing exactly once, where  ${\lambda^+}$ ranges over $\mathrm{HS}(\lambda, n)$.
\end{enumerate} \end{theorem}
\begin{proof} (1) From \cite{Do} or \cite{Bo}  we know that the $P_{\Bbbk}(n,n)$-module $L_{\lambda'} (V_n) \otimes Sym_{n-r}(V_n)$ has a  filtration by $P_{\Bbbk}(n,n)$-modules \[0 = N_t \subseteq N_{t-1} \subseteq \cdots \subseteq N_0=L_{\lambda'}(V_n) \otimes Sym_{n-r}(V_n)\] with factors the Schur modules $L_{\nu}(V_n)$ each appearing exactly once, where $\nu$ ranges over all partitions of $n$ containing $\lambda '$ such that the skew shape $\nu / \lambda'$ is a vertical  strip. Applying the Schur functor $f:P_{\Bbbk}(n,n) \to \m \mathfrak{S}_n$ we have a Specht filtration of $f(L_{\lambda'} (V_n) \otimes Sym_{n-r}(V_n) )$ according to Lemma \ref{basicl}
	\[0=f(N_t) \subseteq f(N_{t-1}) \subseteq \cdots \subseteq f(N_0)=f(L_{\lambda'} (V_n) \otimes Sym_{n-r}(V_n) )\] and the factors of the last filtration are $S^{\lambda^+}$ each appearing exactly once, where $\lambda^+$ ranges over $\mathrm{HS}(\lambda, n)$. From Theorem \ref{main1} we have $f(L_{\lambda'} (V_n) \otimes Sym_{n-r}(V_n) ) \cong  \G_{\Mo(n)}(L_{\lambda'}(V_n))$ as $\mathfrak{S}_n$-modules.
	
	(2) The proof of (2) is similar starting with a Weyl filtration of the $P_{\Bbbk}(n,n)$-module $K_\lambda(V_n) \otimes D_{n-r}(V_n)$.
\end{proof}

\begin{remark} Solomon \cite[Corollary 3.15]{Sol} obtained a branching rule for the restriction of an irreducible for the rook monoid $\IS_n$ to the symmetric group $\mathfrak{S}_n$ over a field of characteristic zero. Also, from Putcha \cite[Theorem 2.1(ii)]{Pu} an analogous result follows for the restriction of an irreducible of the full transformation monoid $\T_n$ that corresponds to a partition that is not a column over a field of characteristic zero. Theorem \ref{br1} generalizes these  via filtrations to any monoid $\Mo(n)$ satisfying $\mathfrak{S}_n \subseteq \Mo(n) \subseteq \PT_n$ and any infinite field $\Bbbk$.
\end{remark}
\subsection{An application of the first branching rule} As an application of Theorem \ref{br1}, we can show that when the characteristic of $\Bbbk$ is zero, the $\Mo(n)$-modules 
$\R^\lambda$ and $\R^\mu$ are not isomorphic for almost all distinct partitions $\lambda$ and $\mu$. To be precise we have the following result.

\begin{corollary}\label{dist}Suppose $\Bbbk$ is a field of characteristic zero and $\Mo(n)$ is a submonoid of $\PT_n$ containing $\mathfrak{S}_n$. Let $\lambda$, $\mu$ be distinct  partitions of $r$, $s$ respectively, where $r,s \le n$. Then the $\Mo(n)$-modules $\R^\lambda$ and $\R^\mu$ are not isomorphic except possibly for the pair $(\lambda, \mu)=((r), (n-r))$. \end{corollary}
\begin{proof} (1)  By Definition \ref{Specht} and Corollary \ref{cormain1}, it suffices to show that the irreducible decompositions of the $\mathfrak{S}_n$-modules $f(L_{\lambda'}(V_n) \otimes Sym_{n-r}(V_n))$ and $f(L_{\mu'}(V_n) \otimes Sym_{n-s}(V_n))$ are distinct, where  $f:P_{\Bbbk}(n,n) \to \m \mathfrak{S}_n$ is the Schur functor. Equivalently, we will show that the irreducible decompositions of the $\M_n(\Bbbk)$-modules $L_{\lambda'}(V_n) \otimes Sym_{n-r}(V_n)$ and $L_{\mu'}(V_n) \otimes Sym_{n-s}(V_n)$ are distinct. This will be done using nothing more than Pieri's rule (see \cite[Section 2.2, eq. (6)]{F} or \cite[(2.3.5) Corollary]{W}) according to which the $M_n({\Bbbk})$-module $L_{\lambda'}(V_n) \otimes Sym_{n-r}(V_n)$ is isomorphic to $\bigoplus_\nu L_\nu(V_n)$, where $\nu$ runs over all partitions of $n$ such that $\nu / \la'$ is a vertical strip of $n-r$ boxes.
	
	Let us assume that at least one of the partitions $\lambda$ and $\mu$ is not a row.
	
For the partitions $\lambda'$ and $\mu'$, let $j$ be the least positive integer such that $\lambda_j' \neq \mu_j'$. Without loss of generality we may assume $\lambda_j' < \mu_j'.$
	
	Case 1. Let $\mu_j' \ge 2$. Consider the partition $(\lambda', 1^{n-r})$ obtained by appending $n-r$ rows to the bottom of $\lambda'$ each of length 1. Then according to Pieri's rule, the Schur module $L_{(\lambda', 1^{n-r})}(V_n)$ is a summand of $L_{\lambda'}(V_n) \otimes  Sym_{n-r}(V_n)$. By the same rule,  $L_{(\lambda', 1^{n-r})}(V_n)$ is not a summand of $L_{\mu'}(V_n) \otimes  Sym_{n-s}(V_n)$ because the $j$th row of the partition of each irreducible summand of  $L_{\mu'}(V_n) \otimes  Sym_{n-s}(V_n)$ has length  at least $\mu'_j$, while the $j$th row of $(\lambda', 1^{n-r})$ has length $\lambda_j' < \mu_j'$.
	
	Case 2. Let $\mu_j'=1$. Then $\lambda_j'=0$. Let $k$ be the largest positive integer such that $\mu_k' \ge 2$. Such a $k$ exists, since otherwise we would have $\mu'=(1^s)$. This implies by the definition of $j$ that  $\lambda'=(1^r)$, so that both of $\lambda'$ and $\mu'$ are columns, contradicting the hypothesis.
	
	For the length of $\lambda'$ we have $\ell(\lambda')=j-1 < j \le \ell(\mu')$ and thus $\ell(\lambda') <  \ell(\mu').$
We continue by distinguishing two subcases.

Case 2a. Suppose $n-r<\ell(\mu')$. Consider the partition $\nu$ obtained by adding one box to each of the first $n-r$ rows of $\lambda'$. Then $L_\nu$ is a summand of $L_{\lambda'}(V_n) \otimes Sym_{n-r}(V_n)$ by Pieri's rule. However, $L_\nu$ is not a summand of $L_{\mu'}(V_n) \otimes  Sym_{n-s}(V_n)$ because the length of the partition $\nu$ is equal to \[\ell(\nu)= \max\{\ell(\lambda'), n-r\}=\max\{j-1, n-r\} < \ell(\mu'),\]
while the length of the partition of each irreducible summand of $L_{\mu'}(V_n) \otimes  Sym_{n-s}(V_n)$ is  at least $\ell(\mu')$ according to Pieri's rule.

Case 2b. Suppose $n-r \ge \ell(\mu')$. From the definition of $j$ we have $|\mu'|=|\lambda'|+\ell(\mu')-j+1$. So $s=r+\ell(\mu')-j+1$. Now using the hypothesis of Case 2b we obtain $n-s=n-r-\ell(\mu')+j-1 \ge j-1$
and thus \begin{equation}\label{inj}n-s \ge j-k.\end{equation}
Consider the partition $\xi$ obtained by adding one box to each of the rows $k+1, k+2, ..., k+n-s$ of $\mu'$, cf. Figure 1.	(Recall that $\mu_k' \ge2$ and thus $\xi$ is indeed a partition). From (\ref{inj}) we have $k+n-s \ge j$. This means the $j$th row of $\xi$ has length $\mu'_j + 1=2$. From Pieri's rule, $L_\xi(V_n)$ is a summand of $L_{\mu'}(V_n) \otimes Sym_{n-s}(V_n)$. However, $L_\xi(V_n)$ is not a summand of $L_{\lambda'}(V_n) \otimes  Sym_{n-r}(V_n)$ because the length of the $j$th row of the partition of each irreducible summand of $L_{\lambda'}(V_n) \otimes  Sym_{n-r}(V_n)$ is at most $\lambda_j'+1=1$ by Pieri's rule.
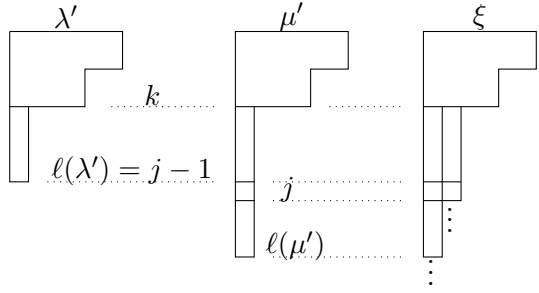
\begin{figure}
	\begin{tikzpicture}[scale=0.5]
		\draw (0,0)-- (3,0)--(3,-1)--(2,-1)--(2,-2)--(0,-2)--(0,0);
		\draw (0,-2)-- (0,-4)--(0.5,-4)--(0.5,-2);
		\draw (6,0)-- (9,0)--(9,-1)--(8,-1)--(8,-2)--(6,-2)--(6,0);
		\draw (6,-2)-- (6,-4)--(6.5,-4)--(6.5,-2);
		\draw (6,-4)--(6,-4.5)--(6.5,-4.5)--(6.5,-4);
		\draw (6,-4.5)--(6,-6)--(6.5,-6)--(6.5,-4.5);
		\draw [dotted] (1,-4)--(5.5, -4);
		\node (1) at (3.25, -3.7) {$\ell(\lambda')=j-1$};
		\draw [dotted] (7,-4.5)--(10.5, -4.5);
		\node (2) at (7.4, -4.2) {$j$};
		\draw [dotted] (2.7,-2)--(5.5, -2);
		\node (3) at (3.8, -1.7) {$k$};
		\node (4) at (7.6, -5.7) {$\ell(\mu')$};
		\draw [dotted] (7,-6)--(10.5, -6);
		\node (5) at (1.5, 0.4) {$\lambda'$};
		\node (6) at (7.5, 0.4) {$\mu'$};
		\draw (11,0)-- (14,0)--(14,-1)--(13,-1)--(13,-2)--(11,-2)--(11,0);
		\draw (11,-2)-- (11,-4)--(11.5,-4)--(11.5,-2);
		\draw (11,-4)--(11,-4.5)--(11.5,-4.5)--(11.5,-4);
		\draw (11,-4.5)--(11,-6)--(11.5,-6)--(11.5,-4.5);
		\draw [dotted] (7,-4)--(10.5, -4);
		\draw [dotted] (8.5,-2)--(10.5, -2);
		\draw (12,-2)--(12,-4.5)--(11.5,-4.5);
		\node (7) at (11.7,-4.8) {$\vdots$};
		\node (7) at (11.2,-6.2) {$\vdots$};
		\draw (12,-4)--(11.5,-4);
		\node (9) at (12.5, 0.4) {$\xi$};
	\end{tikzpicture}
	\caption{The partitions $\lambda'$, $\mu'$ in Case 2 and the partition $\xi$ in Case 2b.}
\end{figure}
We have shown part (1) of the Corollary under the assumption that at least one of $\lambda$ and $\mu$ is not a  row. Assume now that each of
$\lambda$ and $\mu$ is a  row, $\lambda=(r)$ and $\mu=(s)$, where $r,s$ satisfy $r,s \le n$ and $r \neq s$ and $r \neq n-s$. From Theorem \ref{bm} we have $\dim \R^{(r)} = \tbinom{n}{r}, \ \dim \R^{(s)} = \tbinom{n}{s}.$ From the assumptions on $r,s$ it follows that $\tbinom{n}{r} \neq \tbinom{n}{s}$ and thus the modules $\R^{(r)}$ and $\R^{(s)} $ are not isomorphic.\end{proof}

\begin{remark}Suppose $\Mo(n)$ is the symmetric inverse monoid $\IS_n$ or the partial transformation monoid $\PT_n$. For $r\le n $ consider the one rowed partitions $\lambda = (r)$ and $\mu=(n-r)$ and assume that these are distinct. We will show that the corresponding $\Mo(n)$-modules $\R^{(r)}$ and $\R^{(n-r)}$ (for which Corollary \ref{dist} provides no information) are not isomorphic. Indeed, let us assume $r<n-r$. Consider the diagonal matrix $A=\diag(1^r,0^{n-r}) \in  \IS_n $. For the canonical basis $\{e_1,\dots, e_n\}$ of the natural module of $\Mo(n)$ we have $Ae_i=e_i$ for all $i\le r$, and $Ae_j=0$ for all $j\ge r+1$. 
	
		A basis of $\R^{(r)}$ (respectively $\R^{(n-r)}$) is given by the cosets of
		the square-free monomials in $Sym_r(V_n)$ (respectively $Sym_{n-r}(V_n)$)
		in $e_1,\dots,e_n$ of degree $r$ (respectively $n-r$), by Theorem \ref{bm}.
		
		Since $r<n-r$, every square-free monomial of degree $n-r$ in
		$e_1,\dots,e_n$ contains at least one factor $e_j$ with $j\ge r+1$. Hence $A$
		sends every basis element of $\R^{(n-r)}$ to zero. Thus $A$ acts as zero on
		$\R^{(n-r)}$. On the other hand, the action of $A$ on $\R^{(r)}$ is nonzero,
		since $A$ fixes the coset of $e_1\cdots e_r$. Thus the $\Mo(n)$-modules
		$\R^{(r)}$ and $\R^{(n-r)}$ are not isomorphic.
\end{remark}

\subsection{Second branching rule}
Let $\Mo(n)$ be one of the monoids $\IS_n, \PT_n$ or $\T_n$. For $1 \le s < n$ consider the monomorphism of monoids $\Mo(s) \times \Mo(n-s) \to \Mo(n)$
\[(A,B) \mapsto
\left(
\begin{array}{c|c}
A	&   \textbf{0} \\
	\hline
	\textbf{0}&B 
\end{array}
\right),\]
where the matrix $A$ (respectively, $B$) is located in the upper left $s \times s$ corner (respectively, lower right $(n-s) \times (n-s)$ corner). Thus every $\Mo(n)$-module can be considered an $\Mo(s) \times \Mo(n-s)$-module. In Section \ref{49}  we will prove the following branching rule.

\begin{theorem}[Branching rule for $\Mo(s) \times \Mo(n-s) \to \Mo(n)$]\label{br2} Let $\Mo(n)$ be one of the monoids $\IS_n, \PT_n$ or $\T_n$ and let $\lambda / \mu$ be a skew partition of $r$, where $1 \le r \le n$. Then as an $\Mo(s) \times \Mo(n-s)$-module, $\R(n)^{\lambda / \mu}$ (respectively $\R(n)_{\lambda / \mu}$) has a  filtration with factors $\R(s)^{\nu / \mu} \otimes \R(n-s)^{\lambda / \nu}$ (respectively $\R(s)_{\nu / \mu} \otimes \R(n-s)_{\lambda / \nu}$) each appearing exactly once, where ${\nu}$ ranges over all partitions satisfying \begin{center}
		$\mu \subseteq \nu \subseteq \lambda$, \ $|\nu / \mu | \le s$ \ and \ $|\lambda / \nu | \le n-s$.
	\end{center}
\end{theorem}

\subsection{$GL_n$ branching rule}For the proof of Theorem \ref{br2} we will apply the branching rule for the general linear group $\GL_n(\Bbbk)$ which we now recall. We follow closely \cite[p. 238--241]{ABW}. 

Let us consider the lexicographic order of partitions. For partitions $\xi=(\xi_1, \dots, \xi_q), \nu= (\nu_1, \dots, \nu_q)$, we write $\xi \ge \nu$  if there is an index $i$ such that $\xi_1=\nu_1, \dots, \xi_{i-1}=\nu_{i-1}$ and $\xi_i \ge \nu_i$. We write $\xi > \nu$ if $\xi \ge \nu$ and $\xi \neq \nu$.

Let $\lambda / \mu$ be a skew partition of $r$, where $r \le n$. Consider the direct sum decomposition of the natural $\GL_n(\Bbbk)$-module $V_n$ as vector spaces\[V_n= F \oplus G\] where $F$ and $G$ are the subspaces of $V_n$ with basis $e_1, \dots, e_s$ and $e_{s+1}, \dots, e_n$ respectively. 

If $a, b$ are nonnegative integers, we have the map \[\Lambda^a(F) \otimes \Lambda^b(G) \to \Lambda^{a+b}(V_n)\] that sends $f \otimes g \mapsto fg$. For a partition $\nu$ satisfying $\mu' \subseteq \nu \subseteq \lambda'$ we obtain a map \[\Lambda^{\nu / \mu'}(F) \otimes \Lambda^{\lambda'/\nu}(G) \to \Lambda^{\lambda'/\mu'}(V_n)\]
by taking the tensor product of the previous maps $\Lambda^{\nu_i-\mu'_i}(F) \otimes \Lambda^{\lambda'_i-\nu_i}(G) \to \Lambda^{\lambda'_i-\mu'_i}(V_n)$ for $i=1,\dots, \ell(\lambda' / \mu')$, where $\ell(\lambda' / \mu')$ is the number of parts of $\lambda'/\mu'$. Define subspaces ${Q}_\nu$ and $\dot{Q}_\nu$ of $\Lambda^{\lambda'/\mu'}(V_n)$ as follows 
\begin{align}\label{defQ}
Q_\nu&= \Ima\big(\sum_{\substack{\mu' \subseteq \xi \subseteq \lambda'\\\xi \ge \nu}}\Lambda^{\xi /\mu'}(F) \otimes \Lambda^{\lambda' /\xi}(G) \to \Lambda^{\lambda'/\mu'}(V_n)\big),\\
		\dot{Q}_\nu&= \Ima \big(\sum_{\substack{\mu' \subseteq \xi \subseteq \lambda'\\\xi > \nu}}\Lambda^{\xi/ \mu'}(F) \otimes \Lambda^{\lambda'/\xi}(G) \to \Lambda^{\lambda'/\mu'}(V_n)\big),
\end{align}
where in the first sum $\xi$ ranges over all partitions such that $\mu' \subseteq \xi \subseteq \lambda'$ and $\xi \ge \nu$, and in the second sum $\xi$ ranges over all partitions such that $\mu' \subseteq \xi \subseteq \lambda'$ and $\xi > \nu$.
Finally, define subspaces ${M}_\nu$ and $\dot{M}_\nu$ of the Schur module $L_{\lambda'/\mu'}(V_n)$ as follows 
\begin{align*}
	M_\nu&= d_{\lambda'/\mu'}(Q_\nu),\\
\dot{M}_\nu&= d_{\lambda'/\mu'}(\dot{Q}_\nu).
\end{align*}
(With the notation of \cite[Definition II.4.8]{ABW}, our $M_\nu$ is equal to $M_\nu (L_{\lambda' / \mu'}(F\bigoplus G))$).

It follows from the definition that if $\nu, \xi$ are partitions such that $\mu' \subseteq \nu \subseteq \lambda'$, $\mu' \subseteq \xi \subseteq \lambda'$ and $\nu \subseteq \xi$, then $M_\xi \subseteq M_\nu$. Hence we have a filtration  \begin{equation}\label{fil3}0 \subseteq M_{\lambda'} \subseteq \cdots \subseteq M_{\nu} \subseteq \cdots \subseteq M_{\mu'}= L_{\lambda' / \mu'}(V_n)\end{equation}
of $L_{\lambda' / \mu'}(V_n)$ by $\GL_s \times \GL_{n-s}$-submodules. Now according to \cite[Theorem II.4.11]{ABW} we have an isomorphism of $\GL_s \times \GL_{n-s}$-modules \begin{equation}\label{iso3}L_{\nu / \mu'}(F) \otimes L_{\lambda' / \nu}(G) \cong M_\nu / \dot{M}_{\nu}.\end{equation}
Moreover the isomorphism in (\ref{iso3}) sends $d_{\nu / \mu'}(F)(X_{T_1}) \otimes d_{\lambda' / \nu}(G)(X_{T_2})$  to $d_{\lambda' / \mu'}(V_n)(X_{T}) + \dot{M}_{\nu}$, where the $i$th row of $T$ is  
\begin{center}
	\ytableausetup
{smalltableaux, mathmode, boxframe=normal, boxsize=1.6em}\begin{ytableau}
	a_{1} & \cdots & a_{u} &b_1 & \cdots & b_{v}
\end{ytableau}
\end{center}
if the $i$th rows of $T_1$ and $T_2$ are respectively 

\begin{center}
	\ytableausetup
	{smalltableaux, mathmode, boxframe=normal, boxsize=1.6em}\begin{ytableau}
		a_{1} & \cdots & a_{u} 
	\end{ytableau}, \ \ytableausetup
	{smalltableaux, mathmode, boxframe=normal, boxsize=1.6em}\begin{ytableau}
	b_{1} & \cdots & b_{v} 
	\end{ytableau},
\end{center}
for $i=1, \dots, \ell(\lambda' / \mu')$, where $u=\nu_i - \mu'_i$ and $v=\lambda'_i - \nu_i$.
\begin{example*}\label{GLbranchingexample}
	We illustrate the filtration in (\ref{fil3}) in a small case. Let 
$
	\lambda=(2,2,1),\ \mu=(1).
$
	Then
	$
	\lambda'=(3,2),\ \mu'=(1).
$
Let
	\(n=7\) and write
	\[
	V_7=F\oplus G,
	\qquad
	F=\langle e_1,e_2,e_3,e_4\rangle,
	\qquad
	G=\langle e_5,e_6,e_7\rangle .
	\]
	The partitions \(\nu\) satisfying
	$
	\mu'\subseteq \nu\subseteq \lambda'
$ are
$
	(3,2),\ (3,1),\ (3),\ (2,2),\ (2,1),\ (2),\ (1,1),\ (1)
	$ and we have the filtration \[0 \subseteq M_{(3,2)} \subseteq
	M_{(3,1)}\subseteq M_{(3)} \subseteq M_{(2,2)} \subseteq M_{(2,1)} \subseteq M_{(2)} \subseteq M_{(1,1)} \subseteq M_{(1)}=L_{(3,2)/(1)}(V_7).\]

	For $\nu=\lambda'=(3,2)$ there is only one $\xi$ to consider in eq. (\ref{defQ}), namely $ \xi= \lambda'=(3,2)$, and thus we have \begin{align*} Q_{(3,2)} = \Ima\big(\Lambda^{(3,2) /(1)}(F)\otimes \Lambda^0(G)  \to \Lambda^{(3,2) /(1)}(V_n)\big).\end{align*}
	Hence $M_{(3,2)}$ is the subspace of $L_{(3,2)/(1)}(V_7)$ spanned by the elements $d_{(3,2)/(1)}(X_T)$, where $T$ runs over the tableaux 	\begin{equation}\label{tab(3,2)} \begin{ytableau}
		\none & f_1 &  f_2\\
		f_3 & f_4
	\end{ytableau} \ \  \text{such that}  \ f_i \in \{1,2,3,4\}.\end{equation}
	For $\nu=(3,1)$ there are two partitions $\xi$ to consider in eq. (\ref{defQ}), namely $ \xi= \lambda'=(3,2)$ and $ \xi=\nu=(3,1)$, and thus we have 
	\begin{equation*} Q_{(3,1)} = \Ima\big(\Lambda^{(3,2) /(1)}(F)\otimes\Lambda^{0}(G) + \Lambda^{(3,1)/(1)}(F)\otimes \Lambda^{(3,2)/(3,1)}(G) \to \Lambda^{(3,2) /(1)}(V_n)\big).\end{equation*}
Hence $M_{(3,1)}$ is the subspace of $L_{(3,2)/(1)}(V_7)$ spanned by the elements $d_{(3,2)/(1)}(X_T)$, where  $T$ runs over the tableaux (\ref{tab(3,2)}) and over the tableaux
\begin{equation}\label{tab(3,1)} \begin{ytableau}
		\none & f_1 &  f_2\\
		f_3 & g
	\end{ytableau} \ \  \text{such that}  \ f_i \in \{1,2,3,4\}, \ g \in \{5,6,7\}.\end{equation}
For $\nu=(3,1)$, the isomorphism (\ref{iso3}) is induced by the map \begin{align*}L_{(3,1)/(1)}(F) \otimes L_{(3,2)/(3,1)}(G) &\to L_{(3,2)/(1)}(V_7),\\  d_{(3,1)/(1)}(X_{T_1}) \otimes d_{(3,2)/(3,1)}(X_{T_2}) &\mapsto d_{(3,2)/(1)}(X_{T}), \end{align*}
where, if  \begin{align*}&T_1=\begin{ytableau}
	\none & f_1 &  f_2\\
	f_3 
\end{ytableau} \in \tab_{(3,1)/(1)}([4]),\\ \ \ 
&T_2=\begin{ytableau}
	\none & \none & \none \\  \none & g
\end{ytableau} \in \tab_{(3,2)/(3,1)}(\{5,6,7\}), \end{align*}
then $T$ is defined by \[T=\begin{ytableau}
	\none & f_1 &  f_2\\
	f_3 &g \end{ytableau}\in \tab_{(3,2)/(1)}([7]).
\]

\end{example*}

\subsection{Proof of Theorem \ref{br2}}\label{49}
With the previous preparation, we may now prove the branching rule associated with  $\Mo(s) \times \Mo(n-s) \to \Mo(n)$.
\begin{proof}[Proof of Theorem \ref{br2}] We adopt the notation and assumptions of Theorem \ref{br2}. From the filtration (\ref{fil3}), we obtain the filtration  \begin{equation}\label{fil4}0 \subseteq \frac{M_{\lambda'}+N }{N}\subseteq \cdots \subseteq \frac{M_{\nu}+N}{N} \subseteq \cdots \subseteq \frac{M_{\mu'}+N}{N}= \frac{L_{\lambda' / \mu'}(V_n)}{N} = \R(n)^{{\lambda / \mu}}\end{equation}
of 	$\R(n)^{{\lambda / \mu}}$ by \(\Mo(s)\times \Mo(n-s)\)-submodules, where $N=\bigoplus_{\alpha \in \Lambda(n,r)''}(L_{\lambda' / \mu'}(V_{n}))^\alpha$. Using the map (\ref{iso3}) we have the following composition  \begin{equation}\label{iso4}\beta_\nu :L_{\nu / \mu'}(F) \otimes L_{\lambda' / \nu}(G) \cong M_\nu / \dot{M}_{\nu} \to \frac{M_\nu +N}{\dot{M}_{\nu}+N},\end{equation}
where the map on the right is induced by the inclusion $M_\nu \to M_\nu + N$. The map $\beta_\nu$ is clearly surjective.  Consider an element $v \in L_{\nu / \mu'}(F) \otimes L_{\lambda' / \nu}(G)$ of the form $v=d_{\nu / \mu'}(F)(X_{T_1}) \otimes d_{\lambda' / \nu}(G)(X_{T_2})$. If the tableau $T_1$ has an entry of weight greater than or equal to 2, then so does the tableau $T$ defined after (\ref{iso3}). Hence we have $d_{\lambda' / \mu'}(V_n)(X_{T}) \in N$ according to the definition of $N$. In other words we have $v \in \ker \beta_\nu$. Likewise we conclude that $w \in \ker \beta_\nu$ if $w$ is of the form $w=d_{\nu / \mu'}(F)(X_{T_1}) \otimes d_{\lambda' / \nu}(G)(X_{T_2})$ such that the tableau $T_2$ has an entry of weight greater than or equal to 2. Therefore \[N_1 \otimes L_{\lambda' / \nu}(G) +L_{\nu / \mu'}(F)\otimes N_2 \subseteq \ker\beta_\nu, \] where $N_1$ (respectively, $N_2$)is the subspace of $L_{\nu / \mu'}(F)$ (respectively $L_{\lambda' / \nu}(G)$) spanned by elements  $d_{\nu / \mu'}(F)(X_{T_1})$ (respectively $d_{\lambda' / \nu}(G)(X_{T_2})$ ) such that $T_1$ (respectively $T_2$) has an entry of weight greater than or equal to 2. According to Definition \ref{Specht} we have
	$\R(s)^{\nu' / \mu} = L_{\nu / \mu'}(F) / N_1$  and   $\R(n-s)^{\lambda / \nu'} = L_{\lambda' / \nu}(G) / N_2$. Thus the map $\beta_\nu$ induces a surjective map \begin{equation}\label{bar1} \bar{\beta_{\nu}}: 
	\R(s)^{\nu' / \mu} \otimes \R(n-s)^{\lambda / \nu'} \to \frac{M_\nu +N}{\dot{M}_{\nu}+N}.
\end{equation} We intend to show that $\bar{\beta}_\nu$ is an isomorphism.

By summing with respect to the partitions $\nu$ that satisfy $\mu' \subseteq \nu \subseteq \lambda'$, we obtain a surjective linear map 
\begin{equation}\label{bar2} \bigoplus_{\mu' \subseteq \nu \subseteq \lambda' }\R(s)^{\nu' / \mu}\otimes \R(n-s)^{\lambda / \nu'} \to \bigoplus_{\mu' \subseteq \nu \subseteq \lambda' }\frac{M_\nu +N}{\dot{M}_{\nu}+N} \cong \R(n)^{{\lambda / \mu}}.
\end{equation}
	We compute the dimension of $\bigoplus_{\mu' \subseteq \nu \subseteq \lambda' }\R(s)^{\nu' / \mu}\otimes \R(n-s)^{\lambda / \nu'}$. First we note that for a partition $\nu$ we have $\mu' \subseteq \nu \subseteq \lambda'$ if and only if $\mu \subseteq \nu' \subseteq \lambda$ and thus 
	\[ \dim \big( \bigoplus_{\mu' \subseteq \nu \subseteq \lambda' }\R(s)^{\nu' / \mu}\otimes \R(n-s)^{\lambda / \nu'}\big) = \dim\big(\bigoplus_{\mu \subseteq \nu \subseteq \lambda }\R(s)^{\nu / \mu}\otimes \R(n-s)^{\lambda / \nu}\big).\]
	In the computation that follows, for the first and last equality we use Theorem \ref{bm}, in the second we rearrange terms, in the fourth we apply the branching rule identity $\sum_{\substack{\mu \subseteq \nu \subseteq \lambda \\ |\nu|=|\mu|+t }} f^{\nu / \mu}f^{\lambda / \nu} =f^{\lambda / \mu}$ for skew Specht modules, where $t$ is a fixed integer satisfying $0 \le t \le r$,  and in the fifth we use the Chu-Vandermonde identity:
	\begin{align*}&\dim\big(\bigoplus_{\mu \subseteq \nu \subseteq \lambda }\R(s)^{\nu / \mu}\otimes \R(n-s)^{\lambda / \nu}\big)=\sum_{\mu \subseteq \nu \subseteq \lambda }\binom{s}{|\nu / \mu |} f^{\nu / \mu}\binom{n-s}{|\lambda / \nu |} f^{\lambda / \nu}\\&=\sum_{t=0}^r \ \sum_{\substack{\mu \subseteq \nu \subseteq \lambda \\ |\nu|=|\mu|+t }}\binom{s}{|\nu / \mu |} \binom{n-s}{|\lambda / \nu |} f^{\nu / \mu}f^{\lambda / \nu}=
		\sum_{t=0}^r \ \binom{s}{t} \binom{n-s}{r-t} \sum_{\substack{\mu \subseteq \nu \subseteq \lambda \\ |\nu|=|\mu|+t }} f^{\nu / \mu}f^{\lambda / \nu}\\&=
		\sum_{t=0}^r \ \binom{s}{t} \binom{n-s}{r-t} f^{\lambda / \mu}=
		\binom{n}{r}f^{\lambda / \mu}=
		\dim (\R(n)^{\lambda / \mu}).
	\end{align*} From this we conclude that the surjective linear map in (\ref{bar2}) is an isomorphism. This means that the map in (\ref{bar1}) is an isomorphism for every $\nu$. \end{proof}

For $s=n-1$ in Theorem \ref{br2} we obtain the following corollary. 
\begin{corollary}[Branching rule for $\Mo(n-1) \subseteq \Mo(n)$]\label{br3} Let $\Mo(n)$ be one of the monoids $\IS_n, \PT_n$ or $\T_n$ and let $\lambda / \mu$ be a skew partition of $r$. Let $P(\lambda, \mu)$ be the set of partitions $\nu$ such that $\mu \subseteq \nu \subseteq \lambda$ and $|\nu|=|\lambda|-1$.
	\begin{enumerate}[leftmargin=*]
		\item[\textup{(1)}] Suppose $r=n$. Then as an $\Mo(n-1)$-module, $\R(n)^{\lambda / \mu}$ (respectively $\R(n)_{\lambda / \mu}$) has a  filtration with factors $\R(n-1)^{\nu / \mu}$ (respectively $\R(n-1)_{\nu / \mu}$) each appearing exactly once, where ${\nu}$ ranges over $P(\lambda,\mu)$.
		\item[\textup{(2)}] Suppose $r<n$. Then as an $\Mo(n-1)$-module, $\R(n)^{\lambda / \mu}$ (respectively $\R(n)_{\lambda / \mu}$) has a  filtration with factors $\R(n-1)^{\nu / \mu}$ (respectively $\R(n-1)_{\nu / \mu}$)  each appearing exactly once, where ${\nu}$ ranges over the set $P(\lambda,\mu) \cup \{\lambda\}$.
	\end{enumerate}
\end{corollary}

\begin{remarks} (1) As mentioned in the Introduction, Theorem \ref{br2} generalizes to $\IS_n, \PT_n$ and $\T_n$ the characteristic-free branching rule for Specht modules of James and Peel \cite[3.1 Theorem]{JaPe} associated with the inclusion $\mathfrak{S}_s \times \mathfrak{S}_{n-s} \to \mathfrak{S}_n$. \\(2) For a modular branching rule concerning the irreducibles in positive characteristic of the generalized rook monoid algebras see \cite[Section 3.2]{MaSr}.
\end{remarks}

\section{Cauchy decompositions and orbit harmonics quotients}\label{mar}
As mentioned in items (iii) and (iv) of the Introduction, the purpose of this section is to establish `Cauchy decompositions'  (a) for certain algebras associated to the monoids $\IS_n, \PT_n, \T_n$ (see Theorem \ref{m2}) and (b) for the orbit harmonics quotients of these monoids (see Corollary \ref{m3}).

From Section 4 we will need only Theorem \ref{bm}.

In Section 5.1 we introduce a family of modules that will be used later. In Section 5.2 we compute dimensions of the graded components of the quotient rings whose module structure is studied in Section 5.3. In Section 5.4 we study orbit harmonics quotients.

\subsection{The modules $L_\lambda(S^{(m-1,1)})$}\label{LS} The goal of this subsection is to introduce a family of modules $L_\lambda(S^{(m-1,1)})$ for the symmetric group $\mathfrak{S}_m$ that feature in a Cauchy decomposition associated to the full transformation monoid $\T_m$, see Theorem \ref{m2}(3). We will use the presentation of these modules given in Lemma \ref{presT} below.

Recall from Section \ref{Sm} that for any finite dimensional $\Bbbk$-vector space $V$ and for any partition $\lambda$ we have the Schur module $L_\lambda(V)$. Here, we will take $V$ to be the standard $\mathfrak{S}_m$-module which we denote by $U_m$. To be precise, let $V_m$ be the natural 	$\mathfrak{S}_m$-module. We have a basis $\{e_1, \dots, e_m\}$ of $V_m$ and the action of $\mathfrak{S}_m$ is defined by $\sigma e_i = e_{\sigma (i)}$, where $\sigma \in \mathfrak{S}_m$. The subspace of $V_m$ spanned by the sum $e_1+\dots + e_m$ is an 	$\mathfrak{S}_m$-submodule of $V_m$. The quotient module \[U_m:=V_m / \langle e_1+\cdots+e_m \rangle\] is called the standard module of $\mathfrak{S}_m$. It is well known that $U_m$ is isomorphic to the Specht module $S^{(m-1,1)}$ corresponding to the partition $(m-1,1)$ of $m$. We will not need this remark.

For every $k \ge 0$, the natural surjection $\pi: V_m \to U_m$ gives a surjection of $\mathfrak{S}_m$-modules \begin{equation}\label{pik}\pi_k: \Lambda^k(V_m) \to \Lambda^k(U_m), v_1 \cdots v_k \mapsto \pi(v_1) \cdots \pi(v_k), \ v_i \in V_m.\end{equation} Define $\textbf{e}:=e_1 +\cdots+ e_m.$ Since $\textbf{e} \in \ker\pi$, we have \begin{equation}\label{kerpi}v_1  \cdots v_k \in \ker \pi_k \end{equation} if $v_j=\textbf{e}$ for some $j \in \{1,\dots, k\}$.

We have that the Schur module $L_\lambda (V_m)$ is an $\mathfrak{S}_m$-module for any partition $\lambda$. We now define an $\mathfrak{S}_m$-submodule of $L_\lambda (V_m)$. We adopt the notation of Section \ref{Sm} and in particular recall we have a map $d_{\lambda} (V_m):\La^{\lambda}(V_m) \to Sym_{\lambda' }(V_m)$ whose image is $L_\lambda(V_m)$. In order to have more compact notation, for a tableau $T \in \tab_\lambda([m])$ let us denote the element $d_\lambda(V_m)(X_T)$ of $L_\lambda(V_m)$ by $[T]$.

\begin{definition}\label{W} Let $\lambda$ be a partition. \begin{enumerate}[leftmargin=*]
		\item For a tableau $T \in \tab_\lambda([m])$ and an element $u \in [m]$, let $T_{i,j}[u] \in \tab_\lambda([m])$ be the tableau obtained from $T$ by replacing the $(i,j)$ entry by $u$.
		\item Let $W_\lambda$ be the subspace of $L_\lambda (V_m)$ generated by the elements $\sum_{u=1}^m[T_{i,j}[u]]$ for all $1 \le i \le \ell (\lambda) $ and $1 \le j \le \lambda_i$ and $T \in \tab_\lambda([m])$.
\end{enumerate}\end{definition}
It follows that $W_\lambda$ is an $\mathfrak{S}_m$-submodule of $L_\lambda (V_m)$.
\begin{example} Let $\lambda=(r)$ and $i=1, j=r$. If the entries of $T$ are $a_{1}, \dots, a_{r-1}, a_{r}$, then \[\sum_{u=1}^m[T_{i,j}[u]]=e_{a_1}\cdots e_{a_{r-1}}(\sum_{u=1}^me_u)=e_{a_1}\cdots e_{a_{r-1}}\textbf{e} \in \Lambda^r(V_m).\]\end{example}
\begin{lemma}\label{presT} Let $\lambda$ be a partition. Then the $\mathfrak{S}_m$-modules $L_\lambda (V_m) / W_\lambda$ and $L_\lambda (U_m) $ are isomorphic.\end{lemma}
\begin{proof} The natural projection $\pi: V_m \to U_m$ gives a surjective map of $\mathfrak{S}_m$-modules $L_\lambda(\pi): L_\lambda (V_m) \to L_\lambda (U_m)$ such that \[d_{\lambda}(V_m)(x_1 \otimes \cdots \otimes x_{\ell(\lambda)}) \mapsto d_{\lambda}(U_m)(\pi_{\lambda_1}(x_1) \otimes \cdots \otimes \pi_{\lambda_{\ell(\lambda)}}(x_{\ell(\lambda)})), \] where $x_i \in \Lambda^{{\lambda_i}}(V_m)$ and the maps $\pi_k$ were defined in (\ref{pik}). We claim that $W_{\lambda} \subseteq \ker L_\lambda(\pi)$. Indeed, let $T \in \tab_\lambda([m])$ and let the entry of $T$ in position $(s,t)$ be $T(s,t)$. Then for all $1 \le i \le \ell (\lambda) $ and $1 \le j \le \lambda_i$ we have by Definition \ref{W}(2),
	\begin{align*}\sum_{u=1}^{m}[T_{i,j}[u]]= d_\lambda (V_m)&\big(e_{T(1,1)}\cdots e_{T(1,\lambda_1)}\otimes \cdots \otimes e_{T(i,1)}\cdots e_{T(i, j-1)}\textbf{e}e_{T(i,j+1)}\cdots \\&e_{T(i, \lambda_i)}\otimes \cdots \otimes e_{T(\ell(\lambda),1)} \cdots  e_{T(\ell(\lambda),\lambda_{\ell (\lambda)})} \big)\end{align*}
	and thus the image of this element under the map $L_\lambda(\pi)$ is equal to 	\begin{align*}d_\lambda (U_m)&\big(\pi_{\lambda_1}(e_{T(1,1)}\cdots e_{T(1,\lambda_1)})\otimes \cdots \otimes \pi_{\lambda_i}(e_{T(i,1)}\cdots e_{T(i, j-1)}\textbf{e}e_{T(i,j+1)}\cdots e_{T(i, \lambda_i)})\\&\otimes \cdots \otimes \pi_{\lambda_{\ell(\lambda)}}(e_{T(\ell(\lambda),1)} \cdots  e_{T(\ell(\lambda),\lambda_{\ell (\lambda)})}) \big).\end{align*}
	But this is equal to $0$ since 	$\pi_{\lambda_i}(e_{T(i,1)}\cdots e_{T(i, j-1)}\textbf{e}e_{T(i,j+1)}\cdots e_{T(i, \lambda_i)})=0$ according to (\ref{kerpi}).
	
	Hence $W_\lambda \subseteq \ker L_\lambda(\pi)$ and we have \begin{equation}\label{d1}\dim L_\lambda (V_m) / W_\lambda \ge \dim L_\lambda (U_m).\end{equation}
	
	We claim that the vector space $L_\lambda (V_m) / W_\lambda$ is generated by the elements $[S]+W_\lambda$, where $S$ runs over the semistandard tableaux of shape $\lambda$ whose weight $(\alpha_1, \dots, \alpha_m)$ satisfies $\alpha_m = 0$. Indeed, let $T \in \tab_\lambda$ be a tableau that in position $(i,j)$ has the entry $m$ for some $(i,j)$. Then from the definition of $W_\lambda$ we have \begin{equation}\label{d2}[T]+W_\lambda =-\sum_{u=1}^{m-1}[T_{i,j}[u]]+W_\lambda.\end{equation}
	Applying eq. (\ref{d2}) several times if needed, we see that $[T] + W_\lambda$ is a linear combination of various $[S] + W_\lambda$, where $S \in \tab_\lambda$ has weight  $(\alpha_1, \dots, \alpha_m)$ satisfying $\alpha_m = 0$. By the straightening law, Proposition \ref{strl}, each $[S] \in L_\lambda (V_m)$ can be expressed as a linear combination of $[S_v]$, where $S_v \in \tab_\lambda$ is semistandard and has weight equal to the weight of $S$. This proves the claim.
	
	From the claim we have that the dimension $\dim L_\lambda (V_m) / W_\lambda$ is less than or equal to the number of semistandard tableaux in $\tab_\lambda$ with entries from $[m-1]$. In other words, by Theorem \ref{Bthm} we have \begin{equation}\label{d3}\dim L_\lambda (V_m) / W_\lambda \le \dim L_\lambda (U_m).\end{equation} From (\ref{d1}) and (\ref{d3}) we have equality $\dim L_\lambda (V_m) / W_\lambda = \dim L_{\lambda}(U_m)$ of finite dimensions and thus the map $L_\lambda(\pi): L_\lambda (V_m) \to L_\lambda (U_m)$ induces an isomorphism $ L_\lambda (V_m) / W_\lambda \cong L_\lambda (U_m)$.\end{proof}

\subsection{The ideals $J_{m,n}(Z)$ for $Z \in \{ \IS, \PT, \T \} $}

Let $\textbf{x}_{m,n}$ be an $m \times n$ matrix of commuting variables $(x_{i,j})$, where $1 \le i \le m$ and $1 \le j \le n$, and consider the polynomial ring $\Bbbk[\textbf{x}_{m \times n}]$.
\begin{definition}\label{Iideals}Define the following ideals of $\Bbbk[\textbf{x}_{m \times n}]$.\begin{enumerate}[leftmargin=*]
		\item[\textup{(1)}] Let $J_{m,n}(\IS)$ be the ideal generated by 
		\begin{itemize}
			\item 	any product $x_{i,j}x_{i',j}$ of variables in the same column, $1 \le i,i' \le m$ and $1 \le j \le n$ and
			\item any product $x_{i,j}x_{i,j'}$ of variables in the same row, $1 \le i \le m$ and $1 \le j, j' \le n$.
		\end{itemize}
		\item[\textup{(2)}] Let $J_{m, n}(\PT)$ be the ideal  generated by 
		\begin{itemize}
			\item any product $x_{i,j}x_{i',j}$ of variables in the same column, $1 \le i,i' \le m$ and $1 \le j \le n$.
		\end{itemize}
		\item[\textup{(3)}] Let $J_{m,n}(\T)$ be the ideal  generated by 
		\begin{itemize}
			\item any product $x_{i,j}x_{i',j}$ of variables in the same column, $1 \le i,i' \le m$ and $1 \le j \le n$ and 
			\item the sum $x_{1,j} + x_{2,j}+ \cdots +x_{m,j}$ of all the variables in the same column, $1 \le j \le n$.
		\end{itemize}
		\item[\textup{(4)}] When $m=n$ we denote $J_{n,n}(Z)$ by $J_{n}(Z)$, where $Z \in \{ \IS, \PT, \T \} $.
\end{enumerate}\end{definition}

We have the grading $\Bbbk[\textbf{x}_{m \times n}] = \bigoplus_{r \ge0}\Bbbk[\textbf{x}_{m \times n}]_r$, where  $\Bbbk[\textbf{x}_{m \times n}]_r$ is the subspace of $\Bbbk[\textbf{x}_{m \times n}]$ generated by the homogeneous polynomials of degree $r$. If $I$ is a homogeneous ideal of $\Bbbk[\textbf{x}_{m \times n}]$ we obtain the grading $\Bbbk[\textbf{x}_{m \times n}]/ I= \bigoplus_{r\ge 0} \big(\Bbbk[\textbf{x}_{m \times n}]/ I\big)_r$,  where \[\big(\Bbbk[\textbf{x}_{m \times n}]/ I\big)_r:=\Bbbk[\textbf{x}_{m \times n}]_r/ (\Bbbk[\textbf{x}_{m \times n}]_r \cap I).\] Since the ideals $J_{m, n}(Z)$, where $Z \in \{\PT, \IS, \T\}$, of $\Bbbk[\textbf{x}_{m \times n}]$ are   homogeneous, we have the corresponding grading $\Bbbk[\textbf{x}_{m \times n}]/ J_{m,n}(Z)= \bigoplus_{r\ge 0} \big(\Bbbk[\textbf{x}_{m \times n}]/ J_{m,n}(Z)\big)_r$.

Recall that our convention for binomial coefficients is $\binom{a}{b}=0$ if $b >a$.
\begin{lemma}\label{dim1} We have the following equalities. \begin{enumerate}[leftmargin=*]
		\item[\textup{(1)}]$\dim \big(\Bbbk[\textbf{x}_{m \times n}]/ J_{m, n}(\IS)\big)_r =  \binom{m}{r}\binom{n}{r}r!$.
		\item[\textup{(2)}]$\dim \big(\Bbbk[\textbf{x}_{m \times n}]/ J_{m, n}(\PT)\big)_r =\binom{n}{r}m^r$.
			\item[\textup{(3)}]$\dim \big(\Bbbk[\textbf{x}_{m \times n}]/ J_{m,n}(\T)\big)_r =\binom{n}{r}(m-1)^r.$
			\item[\textup{(4)}]$\dim \big(\Bbbk[\textbf{x}_{n \times n}]/ J_{n}(\PT)\big) = (n+1)^n.$
			\item[\textup{(5)}]$\dim \big(\Bbbk[\textbf{x}_{n \times n}]/ J_{n}(\T)\big) = n^n.$
	\end{enumerate}
\end{lemma}
\begin{proof} (1) We consider eq. (1). From $x_{i,j}x_{i,j'} \in J_{m, n}(\IS)$ it follows that\\ $ x_{i_1,j_1}x_{i_2,j_2} \cdots x_{i_r,j_r} \in J_{m, n}(\IS)$ if $r>m$, since two indices among $i_1, i_2, \dots, i_r \in [m]$ must be equal by the pigeonhole principle. Likewise, from $x_{i,j}x_{i',j} \in J_{m, n}(\IS)$ it follows that $ x_{i_1,j_1}x_{i_2,j_2} \cdots x_{i_r,j_r} \in J_{m, n}(\IS)$ if $r>n$. Suppose $r \le \min\{m,n\}$. Then from the definition of $J_{m, n}(\IS)$ we have that  $\big(\Bbbk[\textbf{x}_{m \times n}]/ J_{m, n}(\IS)\big)_r$ is generated as a vector space by the elements \begin{equation}\label{gen11}x_{i_1,j_1}x_{i_2,j_2} \cdots x_{i_r,j_r} + J_{m, n}(\IS),
	\end{equation}
where $1 \le i_1 < i_2 < \cdots < i_r \le m$, and $j_1, j_2, \dots, j_r \in [n]$ are distinct.

	It is straightforward to verify that the elements in (\ref{gen11}) are linearly independent and hence form a basis of $\big(\Bbbk[\textbf{x}_{m \times n}]/ J_{m, n}(\IS)\big)_r$. We have $\tbinom{m}{r}$ choices for the $i_1, i_2,  \cdots, i_r$ and $\tbinom{n}{r}r!$ choices for the $j_1, j_2, \dots, j_r$. Thus eq. (1) follows.
	
	(2) The proof of (2) is similar: One shows that if $r\le n$, then a basis of  $\big(\Bbbk[\textbf{x}_{m \times n}]/ J_{m, n}(\PT)\big)_r$ is  given by \begin{equation}\label{genJ}x_{i_1,j_1}x_{i_2,j_2} \cdots x_{i_r,j_r} + J_{m,n}(\PT),
	\end{equation}
	where $i_1, i_2, \dots, i_r \in [m]$, and
		$1\le j_1<j_2< \dots< j_r \le n$.

(3) Equation (4) follows from (2) and the binomial expansion of $(n+1)^n$. Likewise,  (5) follows from (3).

(4) Finally we prove (3). Consider the $(m-1) \times n$ matrix $\textbf{x}_{(m-1)\times n} =(x_{i,j})$, where $1 \le i \le m-1$ and $1 \le j \le n$. We have the surjective map of graded $\Bbbk$-algebras $\Bbbk[\textbf{x}_{m \times n}] \to \Bbbk[\textbf{x}_{(m-1) \times n}]$ such that \begin{center}
	$x_{i,j} \mapsto x_{i,j}$ if $i<m$ and $x_{m,j} \mapsto -(x_{1,j}+x_{2,j}+\dots +x_{m-1,j})$.
\end{center} This induces a surjective map of graded $\Bbbk$-algebras \begin{equation}\label{map1}
\Bbbk[\textbf{x}_{m \times n}] / J_{m,n}(\T) \to \Bbbk[\textbf{x}_{(m-1) \times n}] / J_{m-1,n}(\PT)
\end{equation}
since for $1 \le i \le m-1$ we have \begin{align*}x_{i,j}x_{m,j} \mapsto x_{i,j}(-\sum_{t=1}^{m-1}x_{t,j})=-\sum_{t=1}^{m-1}x_{i,j}x_{t,j} \in J_{m-1,n}(\PT) \end{align*}
and likewise for $i=m$ we have 
\begin{align*}x_{m,j}^2\mapsto (-\sum_{t=1}^{m-1}x_{t,j})^2=\sum_{s, t=1}^{m-1}x_{s,j}x_{t,j} \in J_{m-1,n}(\PT). \end{align*} Thus for all $r$ we have \begin{equation}\label{dimin}\dim \big(\Bbbk[\textbf{x}_{m \times n}]/ J_{m,n}(\T)\big)_r  \ge \dim \big(\Bbbk[\textbf{x}_{m-1 \times n}]/ J_{m-1, n}(\PT)\big)_r.\end{equation} On the other hand, it follows from the definition of the ideal $J_{m,n}(\T)$ that the vector space   $\big(\Bbbk[\textbf{x}_{m \times n}]/ J_{m,n}(\T)\big)_r$ is  generated by \begin{equation}\label{gen4}x_{i_1,j_1}x_{i_2,j_2} \cdots x_{i_r,j_r} + J_{m,n}(\T),
\end{equation}
where $i_1, i_2, \dots, i_r \in [m-1]$ and $1\le j_1<j_2< \dots< j_r \le n$.
 Comparing this with (\ref{genJ}) (for $m-1$ in place of $m$), we see that \begin{equation}\label{dimin2}\dim \big(\Bbbk[\textbf{x}_{m \times n}]/ J_{m,n}(\T)\big)_r  \le \dim \big(\Bbbk[\textbf{x}_{m-1 \times n}]/ J_{m-1, n}(\PT)\big)_r.\end{equation} From (\ref{dimin}) we obtain equality in (\ref{dimin2}) and hence (3) follows from (2).
\end{proof}
\subsection{Cauchy decompositions} We consider the polynomial ring $\Bbbk[\textbf{x}_{m \times n}]$ as an $\M_m(\Bbbk) \times \M_n(\Bbbk)$-module by defining \begin{equation}\label{Iact}Ax_{i,j}:=\sum_{k=1}^m a_{k,i}x_{k,j}, \   Bx_{i,j}:=\sum_{k=1}^n b_{k,j}x_{i,k},\end{equation}
for $A=(a_{s,t}) \in \M_m(\Bbbk)$ and $B=(b_{s,t}) \in \M_n(\Bbbk)$, and extending this action multiplicatively and linearly.

Recall from Section \ref{pm} that if $A$ is a matrix in $ \PT_m$, then we have the associated partial map $p_A: [m] \to [m]$. The restriction of the above action to $\PT_m \times \PT_n$ is given as follows. For $(A, B) \in  \PT_m \times \PT_n$ and $1 \le i \le m$, $1 \le j \le n$ we have \[(A,B)x_{i,j} = \begin{cases}x_{p_A(i),p_B(j)}, &i \in \dom(p_A), \ j \in \dom(p_B ),\\
0, & otherwise.  \end{cases}\] For a monomial $x_{i_1, j_1} \cdots x_{i_r, j_r}$ in $\Bbbk[\textbf{x}_{m \times n}]$, we have \begin{equation}\label{act}(A,B)x_{i_1, j_1} \cdots x_{i_r, j_r}=\begin{cases}x_{p_A(i_1), p_B(j_1)} \cdots x_{p_A(i_r), p_B(j_r)}, &i_1, \dots, i_r \in \dom(p_A), \ j_1,\dots, j_r \in \dom(p_B ),\\
	0, & otherwise.\end{cases}\end{equation} 

It is straightforward to verify that, with respect to the above actions, the ideal \begin{itemize}\item $J_{m,n}(\IS)$ is an $\IS_m \times \IS_n$-submodule of $\Bbbk[\textbf{x}_{m \times n}]$,
\item $J_{m,n}(\PT)$ is a $\GL_m(\Bbbk) \times \PT_n$-submodule of $\Bbbk[\textbf{x}_{m \times n}]$, and 
\item $J_{m,n}(\T)$ is an $\mathfrak{S}_m \times \T_n$-submodule of $\Bbbk[\textbf{x}_{m \times n}]$. 
\end{itemize}

\begin{remark}We note that all of the above ideals are $\PT_n$-submodules of $\Bbbk[\textbf{x}_{m \times n}]$.\end{remark}

The graded structure of the corresponding quotient rings is given in Theorem \ref{m2} below, which we regard as the main result of the paper together with Corollary \ref{m3}.
\begin{theorem}\label{m2} Suppose $m,n,r$ are positive integers.
	\begin{enumerate}[leftmargin=*]
		\item[\textup{(1)}] The $\IS_m \times \IS_n$-module $\big(\Bbbk[\textbf{x}_{m \times n}] / J_{m,n}(\IS)\big)_r$ has a filtration with quotients \[{\R(m)}^{\lambda} \otimes {\R(n)}^\lambda, \ \lambda \vdash r, \] each appearing exactly once.
		\item[\textup{(2)}] The $\GL_m(\Bbbk) \times \PT_n $-module $\big(\Bbbk[\textbf{x}_{m \times n}] / J_{m,n}(\PT)\big)_r$ has a filtration with quotients \[L_{\lambda'}(V_m) \otimes {\R(n)}^\lambda, \ \lambda \vdash r,\] each appearing exactly once.
			\item[\textup{(3)}] The $\mathfrak{S}_m \times \T_n $-module, $\big(\Bbbk[\textbf{x}_{m \times n}] / J_{m,n}(\T)\big)_r$ has a filtration with quotients \[L_{\lambda'}(U_m) \otimes {\R(n)}^\lambda, \ \lambda \vdash r,\] each appearing exactly once.
	\end{enumerate}\end{theorem}
Before we give the proof of Theorem \ref{m2}, we make some remarks and recall the Cauchy decomposition of the polynomial ring.
\begin{remarks} (1) If $r>n$ in any of the cases of the above theorem, then $\big(\Bbbk[\textbf{x}_{m \times n}] / J_{m,n}(Z)\big)_r=0$ since $\R(n)^\lambda=0.$ Similarly, in case (1) we may restrict the $r$ to $r \le \min{\{m,n\}}$. In case (2) we may restrict the partitions $\lambda$ so that the first column has length at most $m$, because otherwise $L_{\lambda'}(V_m)=0$. In case (3) we may restrict the partitions $\lambda$ so that the first column has length at most $m-1$.\\
(2) Suppose the characteristic of $\Bbbk$ is equal to zero. 
	\begin{itemize}
\item Since the $\Bbbk$-algebra of the rook monoid is semisimple (see \cite{Mun} or \cite[Theorem 11.5.3]{GaMa}) and the $\IS_m \times \IS_n$-module ${\R(m)}^{\lambda} \otimes {\R(n)}^\lambda$ is irreducible (see \cite{Gro}), part (1) of Theorem \ref{m2} implies that the irreducible decomposition of $\big(\Bbbk[\textbf{x}_{m \times n}] / J_{m,n}(\IS)\big)_r$ is  $\bigoplus_{\lambda   \vdash r}{\R(m)}^{\lambda} \otimes {\R(n)}^\lambda$. This is multiplicity free.
	
\item The $\Bbbk$-algebra of $\PT_n$ is not semisimple for $n>1$ \cite{Ste2}. From Theorem \ref{irr}(1) we see that part (2) of Theorem \ref{m2} gives the composition factors for $\big(\Bbbk[\textbf{x}_{m \times n}] / J_{m,n}(\PT)\big)_r$. This is also multiplicity free. We will see in the proof of Theorem \ref{m2} that an explicit composition series is provided.

\item While ${\R(n)}^\lambda$ is an irreducible module of $\T_n$ when $\lambda \neq (1^r)$, the irreducible decomposition of $L_{\lambda'}(U_m)$ as an $\mathfrak{S}_m$-module is not known in general and is closely related to the so called \textit{restriction problem} that asks for the Specht module multiplicities in the restriction of a Schur module to the symmetric group (e.g. see the Introduction of \cite{OZ}).
 \end{itemize}\end{remarks}

\textbf{Cauchy decomposition of $\Bbbk[\textbf{x}_{m \times n}]$.} For the proof of the above theorem we need to recall the characteristic-free Cauchy decomposition of the polynomial ring $\Bbbk[\textbf{x}_{m \times n}]$ as a $\GL_m(\Bbbk) \times \GL_n(\Bbbk)$-module. There are several approaches due to  D\'esarm\'enien, Kung and Rota \cite{DKR}, to DeConcini, Eisenbud and Procesi \cite{DEP}, to Akin, Buchsbaum and Weyman \cite{ABW} and to Krause \cite[Section 8.4]{Kr2}. We follow closely \cite[Section III.1]{ABW}.

If $ i_1, \dots, i_k \in [m]$ and $ j_1, \dots, j_k \in [n]$, where $k \le \min\{m,n\}$, let us denote by $(i_1, \dots, i_k | j_1, \dots, j_k)$ the $k \times k $ minor of the matrix $\textbf{x}_{m \times n}$ corresponding to rows $i_1, \dots, i_k$ and columns $j_1, \dots, j_k$,
\[(i_1, \dots, i_k | j_1, \dots, j_k):=\det(x_{i_u, j_v}), \ 1\le u, v \le k.\]

\begin{definition}\label{minor}Let $S, T$ be tableaux of shape $\lambda$, where $\lambda_1 \le \min\{m,n\}$. For $1 \le p \le \ell(\lambda)$, let \begin{equation}\label{ST}\langle S|T \rangle _p:= (s_{p,1}, \dots, s_{p,\lambda_p} | t_{p,1}, \dots, t_{p,\lambda_p}),\end{equation}
where the entries of rows $p$ of the tableaux $S$ and $T$ are 
		\[s_{p,1}, s_{p,2}, \dots, s_{p,\lambda_p} \ \text{and} \ 
		t_{p,1}, t_{p,2}, \dots, t_{p,\lambda_p}. \]
Define the product  \begin{equation}\label{ST2}(S|T) :=\prod_{p=1}^{\ell{(\lambda)}}\langle S|T \rangle _p . \end{equation}\end{definition}
Next, we order the partitions of $r$ lexicographically. For $\lambda \vdash r$, let $M_\lambda$ be the subspace of $\Bbbk[\textbf{x}_{m \times n}]$ spanned by the elements $(S|T)$ as $S, T$ run over the tableaux of shape $ \ge  \lambda$ and size $r$. We thus have a filtration of $\Bbbk[\textbf{x}_{m \times n}]_r$, \begin{equation}\label{C1} 0 \subseteq M_{(r)} \subseteq M_{(r-1,1)} \subseteq \dots \subseteq M_{(1^r)}=\Bbbk[\textbf{x}_{m \times n}]_r \end{equation}
by $\M_m(\Bbbk) \times \M_n(\Bbbk)$-submodules. Finally, let $\dot{M}_\lambda$ be the subspace of $\Bbbk[\textbf{x}_{m \times n}]_r$ defined by  $\dot{M}_\lambda :=\sum_{\gamma > \la}{M}_\gamma$, where the sum is over all partitions $\gamma$ of $r$ such that  $ \gamma >  \lambda$.  It is shown in \cite[Corollary III.1.3 and Theorem III.1.4]{ABW} that the map \begin{align}\label{beta} \beta_\lambda: L_\lambda(V_m) \otimes L_\lambda (V_n) &\to M_\lambda/\dot{M}_\lambda, \\\nonumber d_{\lambda}(V_m)(X_S) \otimes d_{\lambda}(V_n)(X_T) &\mapsto ( S| T ) + \dot{M}_\lambda\end{align}
is an isomorphism of $\M_m(\Bbbk) \times \M_n(\Bbbk)$-modules.
\begin{proof}[Proof of Theorem \ref{m2}]\label{proof56}For $Z \in \{\IS, \PT, \T\}$ consider the ideal $J_{m,n}(Z)$ of  $\Bbbk[\textbf{x}_{m \times n}]$. From the filtration (\ref{C1}) we obtain the following filtration 
	 \begin{equation}\label{C2} 0 \subseteq \frac{ M_{(r)}+J_{m,n}(Z)}{J_{m,n}(Z)} \subseteq \frac{ M_{(r-1,1)}+J_{m,n}(Z)}{J_{m,n}(Z)} \subseteq \dots \subseteq \frac{ M_{(1^r)}+J_{m,n}(Z)}{J_{m,n}(Z)}. \end{equation}

Composing the map $\beta_\lambda$ of \eqref{beta} with the natural surjection
\[
M_\lambda/\dot{M}_\lambda
\longrightarrow
\frac{M_\lambda+J_{m,n}(Z)}
{\dot{M}_\lambda+J_{m,n}(Z)},
\]
we obtain a surjective map \begin{equation}\label{beta1}
	L_\lambda(V_m) \otimes L_\lambda (V_n) \to \frac{M_\la +J_{m, n}(Z)}{\dot{M}_\la+J_{m, n}(Z)}.
	\end{equation}

(1) Consider part (1) of the theorem, so $Z=\IS$. We have 

\begin{equation}\label{quiso}\big(\frac{\Bbbk[\textbf{x}_{m \times n}]} {J_{m,n}(\IS)}\big)_r \cong \frac{ M_{(1^r)}+J_{m,n}(\IS)}{J_{m,n}(\IS)} \end{equation} as $\IS_m \times \IS_n$ modules. 

If $r > \min\{m,n\}$, then by Lemma \ref{dim1}(1) we have $\big(\Bbbk[\textbf{x}_{m \times n}]/ J_{m, n}(\IS)\big)_r=0$. Let $r \le  \min\{m,n\}$. Consider a tableau $S$ of shape $\lambda$ with entries from $[m]$ such that the weight of some $a \in [m]$ is at least $2$. We distinguish two cases.

Case 1. Suppose row $q$ of $S$ contains the entry $a$ twice. Then  in the tensor product $\La^{\lambda_1}(V_m)\otimes \cdots \otimes \La^{\lambda_q}(V_m) \otimes \cdots \otimes \La^{\lambda_{\ell(\la)}}(V_m)$ of exterior powers we have  $X_S=0$. Thus $\langle S|T \rangle_q =0$ and $(S|T)=0$.

Case 2. Suppose rows $p$ and $q$ of $S$, where $p \neq q$, contain the entry $a$. Consider the Laplace expansion of the determinant $\langle S|T \rangle _p$ defined in (\ref{ST}) along the row \[x_{a,t_{p,1}} \ x_{a,t_{p,2}} \ \cdots \ x_{a,t_{p,\lambda_p}} .\] This yields \begin{equation}\label{laplace1}\langle S|T \rangle _p =\sum_{j=1}^{\lambda_p} \pm x_{a,t_{p,j}}A_j,
	\end{equation}
where $A_j \in \Bbbk[\textbf{x}_{m \times n}]$. Likewise, we obtain \begin{equation}\label{laplace2}\langle S|T \rangle _q =\sum_{j=1}^{\lambda_q} \pm x_{a,t_{q,j}}B_j,
\end{equation}
where $B_j \in \Bbbk[\textbf{x}_{m \times n}]$ and the entries of row $q$ of $T$ are $t_{q,1}, t_{q,2},\dots, t_{q,\lambda_q}$.

From eqs. (\ref{laplace1}) and (\ref{laplace2}), we see that the product $\langle S|T \rangle _p \langle S|T \rangle _q$ is a linear combination of monomials that have the form $x_{a,t_{p,i}}x_{a,t_{q,j}}C_{i,j},$ where $C_{i,j} \in \Bbbk[\textbf{x}_{m \times n}]$. This means that we have $\langle S|T \rangle _p \langle S|T \rangle _q \in J_{m, n}(\IS)$. Thus $( S|T ) \in  J_{m, n}(\IS)$ since $J_{m, n}(\IS)$ is an ideal.

We have shown in Case 1 and Case 2 that if the tableau $S$ has an entry of weight at least 2, then $( S|T ) \in J_{m, n}(\IS)$ for any tableau $T$ of shape $\lambda$. In a similar way it follows that if the tableau $T$ has an entry of weight at least 2, then $( S|T ) \in J_{m, n}(\IS)$ for any tableau $S$ of shape $\lambda$. Thus we have
 \begin{equation}\label{w2}( S|T ) \in J_{m, n}(\IS)\end{equation} if either of $S$ and $T$ has an entry of weight at least 2. From (\ref{w2}) it follows that the map in  (\ref{beta1}) sends the subspaces
 \[
 (L_\lambda(V_m))''\otimes L_\lambda(V_n)
 \quad\text{and}\quad
 L_\lambda(V_m)\otimes (L_\lambda(V_n))''
 \]
 to \{0\}. Hence from Definitions \ref{Specht} and \ref{mSchurf} we have that (\ref{beta1}) induces a map of $\IS_m \times \IS_n$-modules \[\bar{\beta}_\lambda: {\R(m)}^{\lambda'} \otimes {\R(n)}^{\lambda'} \to \frac{M_\la +J_{m, n}(\IS)}{\dot{M}_\la+J_{m, n}(\IS)}.\]
 The map $\bar{\beta}_\lambda$ is surjective because $\beta_\lambda$ is surjective. By summing with respect to all partitions of $r$ we obtain a surjective map of vector spaces  \begin{equation}\label{surj1}\bigoplus_{\lambda \vdash r} {\R(m)}^{\lambda'} \otimes {\R(n)}^{\lambda'} \to \frac{ M_{(1^r)}+J_{m,n}(\IS)}{J_{m,n}(\IS)}.\end{equation}
 Using Theorem \ref{bm}, the identity $\sum_{\lambda \vdash r}(f^\lambda)^2=r!$, Lemma \ref{dim1}(1) and the isomorphism (\ref{quiso}) we have \begin{align*}&\dim \big(\bigoplus_{\lambda \vdash r} {\R(m)}^{\lambda'} \otimes {\R(n)}^{\lambda'}\big)=\sum_{\lambda' \vdash r}\binom{n}{r}f^{\lambda'} \binom{m}{r}f^{\lambda'}\\
 	&=\binom{m}{r} \binom{n}{r}\sum_{\lambda \vdash r}(f^{\lambda'})^2=\binom{m}{r} \binom{n}{r} r!\\&
 	=\dim \big(\Bbbk[\textbf{x}_{m \times n}]/J_{m, n}(\IS)\big)_r
 	=\dim\big(\frac{ M_{(1^r)}+J_{m,n}(\IS)}{J_{m,n}(\IS)}\big).\end{align*}
Hence the surjective map (\ref{surj1}) is an isomorphism (of vector spaces). This implies that each of the surjective maps $\bar{\beta}_\lambda$ is an isomorphism (of $\IS_m \times \IS_n$-modules). Replacing \(\lambda\) by \(\lambda'\) gives the quotients stated in part (1).

(2) The proof of part (2) of the theorem is very similar to the proof of part (1), the only difference is that the computation of dimensions at the end uses the identity \[m^r=\sum_{\lambda \vdash r}\dim( L_{\lambda'}(V_m))f^\lambda\] (see \cite[eq. (5) on p. 52]{F} for $n$ replaced by $r$) in place of  $r!= \sum_{\lambda \vdash r}(f^\lambda)^2$ and  Lemma \ref{dim1}(2) in place of Lemma \ref{dim1}(1).

(3) The proof of part (3) is also similar to the proof of (1) but needs an extra step. So let $\lambda$ be a partition of $r$. Recall the definition of the subspace $W_\lambda$ of $L_\lambda(V_m)$, cf. Definition \ref{W}. We claim that the vector space $W_\lambda \otimes {\R(n)}^{\lambda'}$ is contained in the kernel of the map  (\ref{beta1}). To prove this, let $S, T \in \tab_\lambda$, where $S$ has entries from $[m]$ and $T$ has entries from $[n]$. Also let $1 \le i \le \ell (\lambda) $ and $1 \le j \le \lambda_i$. According to Lemma \ref{Tlem} below, the polynomial $\sum_{u=1}^{m} (S_{i,j}[u]|T)$ belongs to the ideal $J_{m,n}(\T)$ of  $\Bbbk[\textbf{x}_{m \times n}]$,
where $S_{i,j}[u] \in \tab_\lambda$ is the tableau obtained from $S$ by replacing the $(i,j)$ entry with $u$. Thus from Lemma \ref{presT}, we have that $W_\lambda \otimes {\R(n)}^{\lambda'}$ is contained in the kernel of the map  (\ref{beta1}) as claimed.

Hence from Lemma \ref{presT} it follows that the map ($\ref{beta1}$) induces a map of $\mathfrak{S}_m \times \T_n$-modules \[\tilde{\beta}_\lambda: L_{\lambda}(U_m) \otimes {\R(n)}^{\lambda'} \to \frac{M_\lambda+J_{m, n}(\T)} {\dot{M}_\lambda+J_{m, n}(\T)}.\]
From this point we argue as in the proof of part (1) of the theorem (with some minor  differences). In order to be clear we provide the details.

The map $\tilde{\beta}_\lambda$ is surjective because $\beta_\lambda$ is surjective.   After replacing \(\lambda\) by \(\lambda'\) and summing with respect to all partitions of $r$ we obtain a surjective map of vector spaces \begin{equation}\label{surj3}\bigoplus_{\lambda \vdash r} L_{\lambda'}(U_m) \otimes {\R(n)}^{\lambda} \to \frac{ M_{(1^r)}+J_{m,n}(\T)}{J_{m,n}(\T)}.\end{equation}
Using Theorem \ref{bm}, the identity  \[(m-1)^r=\sum_{\lambda \vdash r}\dim( L_{\lambda'}(U_m))f^\lambda\] (see \cite[eq. (5) on p. 52]{F} for $m$ replaced by $m-1$ and $n$ replaced by $r$) and Lemma \ref{dim1}(3) we have \begin{align}\label{c3}&\dim \big(\bigoplus_{\lambda \vdash r} L_{\lambda'}(U_m) \otimes \R(n)^\lambda\big)=\sum_{\lambda \vdash r}\dim \big(L_{\lambda'}(U_m)\big) \binom{n}{r}f^\lambda\\\nonumber
	&=\binom{n}{r} \sum_{\lambda \vdash r}\dim \big(L_{\lambda'}(U_m)\big) f^\lambda=\binom{n}{r}(m-1)^r \\\nonumber&
	=\dim \big(\Bbbk[\textbf{x}_{m \times n}]/J_{m, n}(\T)\big)_r=\dim\big(\frac{ M_{(1^r)}+J_{m,n}(\T)}{J_{m,n}(\T)}\big).\end{align}
Hence the surjective map (\ref{surj3}) is an isomorphism (of vector spaces). Thus  the surjective map $\tilde{\beta}_\lambda$ is an isomorphism (of $\mathfrak{S}_m \times \T_n$-modules). The proof of part (3) of Theorem \ref{m2} is complete provided we prove  Lemma \ref{Tlem} below.\end{proof}
We adopt the notation of Definition \ref{minor}. 
\begin{lemma}\label{Tlem}Let $S, T \in \tab_\lambda$, where $S$ has entries from $[m]$ and $T$ has entries from $[n]$, and let $1 \le i \le \ell (\lambda) $ and $1 \le j \le \lambda_i$. Then we have \begin{equation}\label{Wker} \sum_{u=1}^{m} (S_{i,j}[u]|T) \in J_{m,n}(\T),
	\end{equation}
	where $S_{i,j}[u] \in \tab_\lambda$ is the tableau obtained from $S$ by replacing the $(i,j)$ entry with $u$.
	\end{lemma}
\begin{proof}From eq. (\ref{ST2}) we have 
	\begin{align*}\sum_{u=1}^{m}(S_{i,j}[u]|T)&=\sum_{u=1}^{m}\langle S_{i,j}[u]|T\rangle_1 \cdots \langle S_{i,j}[u]|T\rangle_i\cdots \langle S_{i,j}[u]|T\rangle_{\ell(\la)}\\&=
	\sum_{u=1}^{m}\langle S|T\rangle_1 \cdots \langle S_{i,j}[u]|T\rangle_i\cdots \langle S|T\rangle_{\ell(\la)}\\&=
	\langle S|T\rangle_1 \cdots \big(\sum_{u=1}^{m}\langle S_{i,j}[u]|T\rangle_i\big)\cdots \langle S|T\rangle_{\ell(\la)},
	\end{align*}
	where in the second equality we used the fact that the tableaux $S$ and $S_{i,j}[u]$ differ only in the $i$th row and thus $\langle S_{i,j}[u]|T\rangle_q = \langle S|T\rangle_q$ for all $q \neq i$.
	Since $J_{m,n}(\T)$ is an ideal of  $\Bbbk[\textbf{x}_{m \times n}]$, it follows that in order to prove (\ref{Wker}), it suffices to show that
	\begin{equation}\label{Wker2} \sum_{u=1}^{m} \langle S_{i,j}[u]|T\rangle_i \in J_{m,n}(\T).\end{equation}

We prove (\ref{Wker2}) as follows. From Definition \ref{minor} we have
\begin{equation}\label{Wker3}
	\sum_{u=1}^{m} \langle S_{i,j}[u]|T\rangle_i
	=
	\sum_{u=1}^{m}
	(s_{i,1}, \dots, s_{i,j-1},u,s_{i,j+1}, \dots, s_{i,\lambda_i}
	\mid
	t_{i,1}, \dots, t_{i,\lambda_i}),
\end{equation}
where $s_{i,q}$, respectively $t_{i,q}$, is the entry of the tableau $S$,
respectively $T$, in position $(i,q)$, and the $u$ on the right hand side is
located in the $j$th position among the row indices. By multilinearity of the
determinant in its $j$th row, the right hand side of (\ref{Wker3}) is the
determinant of the $\lambda_i \times \lambda_i$ matrix whose $q$th row, for
$q\ne j$, is
\[
(x_{s_{i,q},t_{i,1}},x_{s_{i,q},t_{i,2}},\dots,x_{s_{i,q},t_{i,\lambda_i}})
\]
and whose $j$th row is
\[
\bigg(
\sum_{v=1}^m x_{v,t_{i,1}},
\sum_{v=1}^m x_{v,t_{i,2}},
\dots,
\sum_{v=1}^m x_{v,t_{i,\lambda_i}}
\bigg).
\]
Every entry of this $j$th row belongs to $J_{m,n}(\T)$. Expanding the
determinant along this row, we conclude that the determinant belongs to
$J_{m,n}(\T)$, because $J_{m,n}(\T)$ is an ideal of
$\Bbbk[\textbf{x}_{m \times n}]$. Thus we have shown (\ref{Wker2}).\end{proof}

\subsection{Orbit harmonics quotients}\label{oh}
Next we want to relate Theorem \ref{m2} when $m=n$ to the coordinate rings of the affine varieties $\IS_n$, $\PT_n$ and $\T_n$ (see Corollary \ref{m3} below). This will be done using orbit harmonics and thus we need to recall some relevant background. For more details on orbit harmonics we refer to the paper by Rhoades \cite[Section 2.2]{Rho} or the paper by Reiner and Rhoades \cite[Sections 1.1 and 2.5]{RR}. 

If $Z \subseteq \M_{n}(\Bbbk)$ is a finite locus of points in affine $n \times n$-space, we have the vanishing ideal \[\I(Z) := \{f \in \Bbbk[\textbf{x}_{n \times n}]: f(\textbf{z})=0 \ \textup{for all} \ \textbf{z} \in Z\}\]
and an identification of $\Bbbk$-algebras $\Bbbk[Z] \cong  \Bbbk[\textbf{x}_{n \times n}]/\I(Z)$, where $\Bbbk[Z]$ is the algebra of polynomial functions on $Z$. Since $Z$ is a finite set, Lagrange interpolation yields that $\Bbbk[Z]$ is the space of functions $Z \to \Bbbk$. Hence we may regard $\Bbbk[Z]$ as the $\Bbbk$-vector space of $Z$, i.e. the $\Bbbk$-vector space that has basis the points of $Z$.

If $f \in \Bbbk[\textbf{x}_{n \times n}]$ is a nonzero polynomial, we may write $f=f_d+\cdots+f_1+f_0$ where $f_i$ is a homogeneous polynomial of degree $i$ and $f_d \neq 0$. Define $h(f):=f_d,$ the highest degree component of $f$.
\begin{definition}
	Let $ I \subseteq \Bbbk[\textbf{x}_{n \times n}]$ be an ideal of $\Bbbk[\textbf{x}_{n \times n}]$. \textit{The associated graded ideal} $\gr(I)$ of $I$ is the ideal generated by the polynomials $h(f)$, where $f \in I$ and $f \neq 0$.
\end{definition}
  The result we need is the following. \begin{lemma}[\cite{Rho}]Let $Z \subseteq \M_{n}(\Bbbk)$ be a finite locus of points in the affine space
  	of $n\times n$ matrices. Then there is an isomorphism of vector spaces \begin{equation}\label{har1} \Bbbk[Z] \cong \Bbbk[\textbf{x}_{n \times n}] / \I(Z) \cong \Bbbk[\textbf{x}_{n \times n}] / \gr \I(Z).\end{equation}\end{lemma}
For a proof see \cite[Lemma 3.15]{Rho}.

\textbf{The coordinate rings of $\IS_n$, $\PT_n$ and $\T_n$.} Recall we have the ideals $J_{n}(Z)$ of Definition \ref{Iideals}(4) for $Z \in \{\IS, \PT, \T\}$. 

\begin{corollary}\label{main3} For the vanishing ideals $\I(\IS_n)$, $\I(\PT_n)$ and $\I(\T_n)$ of the affine varieties $\IS_n$, $\PT_n$ and  $\T_n$ we have the equalities 
		\[J_{n}(\IS) = \gr(\I(\IS_n)), \ \ J_{n}(\PT) = \gr(\I(\PT_n)), \ \ J_{n}(\T) = \gr(\I(\T_n)).\]
\end{corollary}
\begin{proof} First we show that $J_{n}(\IS) =\gr(\I(\IS_n))$.
	
	For all $1 \le i,j \le n$ the polynomial $x^2_{i,j}-x_{ij}$ vanishes on $\IS_n$ since every entry of a matrix in $\IS_n$ is equal to $0$ or $1$. Also, for all $1 \le i \le n$ and all $1 \le j < j' \le n$ the polynomial $x_{i,j}x_{i,j'}$ vanishes on $\IS_n$ since every row of a matrix in $\IS_n$ has at most one nonzero entry. Thus we have $x^2_{i,j} \in \gr(\I(\IS_n))$ for all $1\le i,j \le n$ and $x_{i,j}x_{i,j'}\in \gr(\I(\IS_n))$ for all $1\le i \le n$ and all $1 \le j < j' \le n$. Hence $x_{i,j}x_{i,j'}\in \gr(\I(\IS_n))$ for all $1 \le i, j, j' \le n$. In a similar way we conclude that $x_{i,j}x_{i',j}\in \gr(\I(\IS_n))$ for all $1 \le i, i', j \le n$. Thus  $J_n(\IS) \subseteq \gr(\I(\IS_n))$.
	
	From (\ref{har1}) and Proposition \ref{card}(1) we have \[\dim \big(\Bbbk[\textbf{x}_{n \times n}] / \gr(\I(\IS_n))\big)= \dim\big(\Bbbk [\IS_n]\big)=|\IS_n|=\sum_{r=0}^n \tbinom{n}{r}^2r!.\] From this and Lemma \ref{dim1}(1) we get \[\dim \big(\Bbbk[\textbf{x}_{n \times n}] / \gr(\I(\IS_n))\big)= \dim \big(\Bbbk[\textbf{x}_{n \times n}] / J_n(\IS)\big).\] Since $J_{n}(\IS) \subseteq \gr(\I(\IS_n))$, it follows that $J_{n}(\IS) = \gr(\I(\IS_n))$.
	
	The proofs of $J_{n}(\PT) = \gr(\I(\PT_n))$ and $J_{n}(\T) = \gr(\I(\T_n))$ are similar to the above proof. \end{proof}
	\begin{corollary}\label{m3}
We have the following decompositions associated to the graded ideals of the vanishing ideals $\I(\IS_n)$, $\I(\PT_n)$ and $\I(\T_n)$ of the varieties $\IS_n$, $\PT_n$ and  $\T_n$.\begin{enumerate}[leftmargin=*]
	\item[\textup{(1)}] As an $\IS_n \times \IS_n$-module, $\big(\Bbbk[\textbf{x}_{n \times n}] / \gr(\I(\IS_n)) \big)_r$ has a filtration with quotients ${\R(n)}^{\lambda} \otimes {\R(n)}^\lambda,$ where $\lambda \vdash r$, each appearing exactly once.
	\item[\textup{(2)}] As a $\GL_n(\Bbbk) \times \PT_n $-module, $\big(\Bbbk[\textbf{x}_{n \times n}]/ \gr(\I(\PT_n))\big)_r$ has a filtration with quotients $L_{\lambda'}(V_n) \otimes {\R(n)}^\lambda,$ where $\lambda \vdash r$, each appearing exactly once.
	\item[\textup{(3)}] As an $\mathfrak{S}_n \times \T_n $-module, $\big(\Bbbk[\textbf{x}_{n \times n}]/ \gr(\I(\T_n))\big)_r$ has a filtration with quotients $L_{\lambda'}(U_n) \otimes {\R(n)}^\lambda,$ where $\lambda \vdash r$, each appearing exactly once. \end{enumerate}\end{corollary}
\begin{proof}This follows from Theorem \ref{m2} and Corollary \ref{main3}.\end{proof}

\appendix
\section{Skew symmetric Cauchy decompositions}\label{ssC}

Our purpose here is to obtain a skew symmetric analog of Theorem \ref{m2}. 
Let us recall that the group $\GL_m({\Bbbk}) \times \GL_n(\Bbbk)$ acts on the exterior algebra $\Lambda(V_m \otimes V_n)$ by extending the usual action on $V_m \otimes V_n$. When the characteristic of $\Bbbk$ is zero, the irreducible decomposition of the degree $r$ homogeneous component  $\Lambda^r(V_m \otimes V_n)$ is $\bigoplus_{\la \vdash r} L_\la(V_m) \otimes K_\la(V_n)$, see \cite[Theorem 4.8.6]{Sag} or \cite[Corollary 2.3.3]{W}. In arbitrary characteristic, the previous decomposition holds up to filtration \cite[(Section III.2) Theorem]{ABW}.
\subsection{Definition and statement}
Let $\textbf{y}_{m,n}=(y_{i,j})$ be an $m \times n$ matrix of skew commuting variables $y_{i,j}$, where $1 \le i \le m$ and $1 \le j \le n$, and consider the corresponding skew polynomial ring $\Bbbk\langle\textbf{y}_{m \times n}\rangle$. So we have \begin{center}
	$y_{i,j}^2=0$ and $y_{i,j}y_{k,l}=- y_{k,l}y_{i,j}$ for all $1 \le i,k \le m$ and $1 \le j, l \le n$.
\end{center}

\begin{definition}\label{skewideals}We define the following ideals of $\Bbbk\langle\textbf{y}_{m \times n}\rangle$.
\begin{enumerate}[leftmargin=*]
	\item[\textup{(1)}] Let $J'_{m,n}(\IS)$ be the ideal  generated by 
	\begin{itemize}
		\item 	any product $y_{i,j}y_{i',j}$ of variables in the same column, $1 \le i,i' \le m$ and $1 \le j \le n$ and
		\item any product $y_{i,j}y_{i,j'}$ of variables in the same row, $1 \le i \le m$ and $1 \le j, j' \le n$.
	\end{itemize}
		\item[\textup{(2)}] Let $J'_{m, n}(\PT)$ be the ideal  generated by 
	\begin{itemize}
		\item any product $y_{i,j}y_{i',j}$ of variables in the same column, $1 \le i,i' \le m$ and $1 \le j \le n$.
	\end{itemize}
	\item[\textup{(3)}] Let $J'_{m,n}(\T)$ be the ideal  generated by 
	\begin{itemize}
		\item any product $y_{i,j}y_{i',j}$ of variables in the same column, $1 \le i,i' \le m$ and $1 \le j \le n$ and 
		\item the sum $y_{1,j} + y_{2,j}+ \cdots +y_{m,j}$ of all the variables in the same column, $1 \le j \le n$.
	\end{itemize}
\end{enumerate}		
\end{definition}

From the definition we have $J'_{m,n}(\PT) \subseteq J'_{m,n}(\T)$.

We have the grading $\Bbbk\langle\textbf{y}_{m \times n}\rangle = \bigoplus_{r \ge0}\Bbbk\langle\textbf{y}_{m \times n}\rangle_r$, where  $\Bbbk\langle\textbf{y}_{m \times n}\rangle_r$ is the subspace of $\Bbbk\langle\textbf{y}_{m \times n}\rangle$ generated by the homogeneous polynomials of degree $r$.  Since the ideals $J'_{m, n}(Z)$, where $Z \in \{\PT, \IS, \T\}$, of $\Bbbk\langle\textbf{y}_{m \times n}\rangle$ are   homogeneous, we have the corresponding grading \[\Bbbk\langle\textbf{y}_{m \times n}\rangle/ J'_{m,n}(Z)= \bigoplus_{r\ge 0} \big(\Bbbk\langle \textbf{y}_{m \times n}\rangle / J'_{m,n}(Z)\big)_r.\]

The formulae in (\ref{Iact}) for $y_{i,j}$ in place of $x_{i,j}$ define actions of $\PT_m \times \PT_n$ and $\GL_m(\Bbbk)$ on $\Bbbk\langle \textbf{y}_{m \times n}\rangle$. It is straightforward to verify that the ideal \begin{itemize}\item $J'_{m,n}(\IS)$ is an $\IS_m \times \IS_n$-submodule of $\Bbbk\langle \textbf{y}_{m \times n}\rangle$,
	\item $J'_{m,n}(\PT)$ is a $\GL_m(\Bbbk) \times \PT_n$-submodule of $\Bbbk\langle \textbf{y}_{m \times n}\rangle$, and 
	\item $J'_{m,n}(\T)$ is an $\mathfrak{S}_m \times \T_n$-submodule of $\Bbbk\langle \textbf{y}_{m \times n}\rangle$. 
\end{itemize}

\begin{lemma}\label{dim1skew} The dimensions of the vector spaces $\big(\Bbbk\langle \textbf{y}_{m \times n}\rangle / J'_{m,n}(Z)\big)_r$, where $Z \in \{\IS, \PT, \T\}$, are as follows. \begin{enumerate}[leftmargin=*]
		\item[\textup{(1)}]	$\dim \big(\Bbbk\langle\textbf{y}_{m \times n}\rangle/ J'_{m, n}(\IS)\big)_r =\binom{m}{r}\binom{n}{r}r!.$
			\item[\textup{(2)}]$\dim \big(\Bbbk\langle\textbf{y}_{m \times n}\rangle/ J'_{m, n}(\PT)\big)_r =\binom{n}{r}m^r.$
			\item[\textup{(3)}]$\dim \big(\Bbbk\langle\textbf{y}_{m \times n}\rangle/ J'_{m,n}(\T)\big)_r =\binom{n}{r}(m-1)^r.$\end{enumerate}
\end{lemma}
\begin{proof} This is very similar to the proof of Lemma \ref{dim1}.
\end{proof}
The main result of the Appendix is the following. \begin{theorem}\label{skewC} Suppose $m,n,r$ are positive integers.
	\begin{enumerate}[leftmargin=*]
		\item[\textup{(1)}] The $\IS_m \times \IS_n$-module $\big(\Bbbk\langle\textbf{y}_{m \times n}\rangle / J'_{m,n}(\IS) \big)_r$ has a filtration with quotients ${\R(m)}^{\lambda'} \otimes {\R(n)}_{\lambda}, \ \lambda \vdash r,$ each appearing exactly once.
		\item[\textup{(2)}] The $\GL_m(\Bbbk) \times \PT_n$-module $\big(\Bbbk\langle\textbf{y}_{m \times n}\rangle/ J'_{m,n}(\PT)\big)_r$ has a filtration with quotients $L_\lambda(V_m) \otimes {\R(n)}_{\lambda}, \ \lambda \vdash r,$ each appearing exactly once.
		\item[\textup{(3)}] The $\mathfrak{S}_m \times \T_n$-module $\big(\Bbbk\langle\textbf{y}_{m \times n}\rangle/ J'_{m,n}(\T)\big)_r$ has a filtration with quotients $L_\lambda(U_m) \otimes {\R(n)}_{\lambda}, \ \lambda \vdash r,$ each appearing exactly once.
	\end{enumerate}
\end{theorem} The proof of Theorem \ref{skewC} will be given in Section \ref{A4} below.

We remark that, to the best of our knowledge, no $\mathfrak{S}_m \times \mathfrak{S}_n$ analog of the above theorem has appeared in the literature, even in characteristic zero.
\subsection{Cauchy decomposition of $\Bbbk\langle\textbf{y}_{m \times n}\rangle$.} For the proof of the above theorem we need to recall the Cauchy decomposition of  $\Bbbk\langle\textbf{y}_{m \times n}\rangle$ as a $\GL_m(\Bbbk) \times \GL_n(\Bbbk)$-module. We follow closely \cite[Section III.2]{ABW}.

\begin{definition}\label{per}For a partition $\lambda$, let $S, T$ be tableaux of shape $\lambda$ with entries from $[m]$ and $[n]$ respectively. 
	
	(1) For $ 1\le p \le \ell(\la)$, suppose the entries in row $p$ of $S$ and row $p$ of $T$ are
			\[s_{1}, \dots, s_{\lambda_p} \ \text{and} \ t_{1},\dots, t_{1},t_{2},\dots, t_{2}, \dots,t_{q},\dots, t_{q},\]
where $t_1, t_2, \dots, t_q$ are distinct.  Define $\langle S|T \rangle_{p} \in  \Bbbk\langle \textbf{y}_{m \times n} \rangle$ by 
\begin{equation}\label{per1}
\langle S|T \rangle_{p}:=\sum_{h_1, ..., h_{\lambda_p}}y_{s_1 , h_1} y_{s_2, h_2} \cdots y_{s_{\la_p}, h_{\la_p}},
\end{equation}
where the sequence of indices $h_1, \dots, h_{\lambda_p}$ runs over all rearrangements of \[t_{1}, \dots, t_{1}, t_{2}, \dots, t_{2}, \dots, t_{q}, \dots, t_{q}.\] 

(2) Define $(S|T) \in  \Bbbk\langle \textbf{y}_{m \times n} \rangle$ by
\begin{equation}\label{per2}(S|T) :=\prod_{p=1}^{\ell(\la)}\langle S|T \rangle_{p}
\end{equation} \end{definition}
We remark that eq. (\ref{per1}) in the previous definition is not the same given in \cite{ABW}, but is easily seen to be equal to the definition given in the fourth line on p. 247 in \cite{ABW}.

Next, we order the partitions of $r$ lexicographically. For $\lambda \vdash r$, let $M_\lambda$ be the subspace of $\Bbbk\langle\textbf{y}_{m \times n}\rangle$ spanned by the elements $(S|T)$ as $S, T$ run over the tableaux of shape $ \ge  \lambda$ and size $r$ with entries from $[m]$ and $[n]$ respectively. We thus have a filtration of $\Bbbk\langle\textbf{y}_{m \times n}\rangle_r$, \begin{equation}\label{skewC1} 0 \subseteq M_{(r)} \subseteq M_{(r-1,1)} \subseteq \dots \subseteq M_{(1^r)}=\Bbbk\langle\textbf{y}_{m \times n}\rangle_r \end{equation}
by $\GL_m(\Bbbk) \times \GL_n(\Bbbk)$-submodules. Finally, let $\dot{M}_\lambda$ be the subspace of $\Bbbk\langle\textbf{y}_{m \times n}\rangle$ spanned by the elements $(S|T)$ as $S, T$ run over the tableaux of shape $ >  \lambda$ and size $r$.  It is shown in \cite[Corollary III.2.3 and Theorem III.2.4]{ABW} that the map \begin{align}\label{betaskew} \beta_\lambda: L_\lambda(V_m) \otimes K_\lambda (V_n) \to M_\lambda/\dot{M}_\lambda,  \ \ d_{\lambda}(V_m)(X_S) \otimes d'_{\lambda}(V_n)(Y_T) \mapsto (S| T ) + \dot{M}_\lambda\end{align}
is an isomorphism of $\M_m(\Bbbk) \times \M_n(\Bbbk)$-modules.

\subsection{Two lemmas} For the proof of Theorem  \ref{skewC} we will need two lemmas concerning the filtration (\ref{skewC1})  and the ideals of Definition \ref{skewideals}.
\begin{lemma}\label{lemskew1} Suppose $\la$ is a partition and $S \in \tab_\lambda([m])$, $T \in \tab_\lambda([n])$. \begin{enumerate}[leftmargin=*]
		\item[\textup{(1)}] If $T$ contains an entry with weight at least 2, then $(S|T) \in J'_{m,n}(\PT) \subseteq J'_{m,n}(\IS)$.
		\item[\textup{(2)}] If $S$ contains an entry with weight at least 2, then $(S|T) \in J'_{m,n}(\IS)$.
	\end{enumerate}
\end{lemma}
\begin{proof} (1) Suppose $j \in [n]$ is an entry of the tableau $T$ with weight at least 2. We distinguish two cases.
	
	Case 1. Suppose $j$ appears twice in the same row of $T$, say row $p$. Since multiplication in the ring $\Bbbk\langle\textbf{y}_{m \times n}\rangle$ is skew commutative, every summand in the right hand side of (\ref{per1}) is of the form  $\pm  y_{i_1,j}y_{i_2,j}f$
	for some $i_1,i_2 \in [m]$ and $f \in \Bbbk\langle\textbf{y}_{m \times n}\rangle$ (depending on the particular summand). This means that $\langle S| T \rangle_p \in J'_{m,n}(\PT)$ since $J'_{m,n}(\PT)$ is an ideal of $\Bbbk\langle\textbf{y}_{m \times n}\rangle$ containing $y_{i_1,j}y_{i_2,j}$ according to Definition \ref{skewideals}(2). From eq. (\ref{per2}) we conclude that $( S | T ) \in J'_{m,n}(\PT)$.
	
	Case 2. Suppose $j$ appears in rows $p$ and $q$ of $T$, where $p \neq q$. We intend to show that the product $\langle S| T \rangle_p \langle S| T \rangle_q$ is in $J'_{m,n}(\PT)$, which implies that $(S|T)$ is in $J'_{m,n}(\PT)$. Indeed, every summand of the right hand side of (\ref{per1}) is of the form  $\pm  y_{i,j}f$
	for some $i \in [m]$ and $f \in \Bbbk\langle\textbf{y}_{m \times n}\rangle$. Likewise every summand of the right hand side of (\ref{per1}) for $q$ in place of $p$ is of the form  $\pm  y_{i',j}f'$
	for some $i' \in [m]$ and $f' \in \Bbbk\langle\textbf{y}_{m \times n}\rangle$. Upon multiplying in $\Bbbk\langle\textbf{y}_{m \times n}\rangle$ we have that every summand of $\langle S| T \rangle_p \langle S| T \rangle_q$ is of the form $\pm  y_{i,j}y_{i',j}ff'$. This is in the ideal $J'_{m,n}(\PT)$.
	
	(2) This is similar to (1).
	\end{proof}
For the next lemma let us recall the following notation from Definition \ref{W}. Let	$S \in \tab_\lambda([m])$ be a tableau and consider the box in position $(i,j)$, where $1 \le i \le \ell (\lambda) $ and $1 \le j \le \lambda_i$. For $u \in [m]$, we denote by  $S_{i,j}[u] \in \tab_\lambda([m])$ the tableau obtained from $S$ by replacing the $(i,j)$ entry with $u$.
\begin{lemma}\label{lemskew2}
Suppose $\la$ is a partition and $S \in \tab_\lambda([m])$, $T \in \tab_\lambda([n])$. Let \(1\le i\le \ell(\lambda)\) and \(1\le j\le \lambda_i\). Then \[\sum_{u=1}^m (S_{i,j}[u]|T) \in J'_{m,n}(\T).\]
\end{lemma}
\begin{proof}	By eq. (\ref{per1}) we have $\langle S_{i,j}[u]|T \rangle_{i}=\sum_{h_1, \dots, h_{\lambda_i}}y_{s_1 , h_1} \cdots y_{u, h_j} \cdots y_{s_{\la_i}, h_{\la_i}}$, where  $h_1, \dots, h_{\lambda_i}$ run over all rearrangements of $t_{1}, \dots, t_{1}, t_{2}, \dots, t_{2}, \dots, t_{q}, \dots, t_{q}$. By summing we obtain  \begin{align}\label{lem21}
		\sum_{u=1}^m\langle S_{i,j}[u]|T \rangle_{i}&=\sum_{u=1}^m\sum_{h_1, \dots, h_{\lambda_i}}y_{s_1 , h_1} \cdots y_{u, h_j} \cdots y_{s_{\lambda_i}, h_{\la_i}} \\\nonumber&=
		\sum_{h_1, \dots, h_{\lambda_i}}y_{s_1 , h_1} \cdots \big(\sum_{u=1}^my_{u, h_j}\big) \cdots y_{s_{\la_i}, h_{\la_i}}
		\in J'_{m,n}(\T),
	\end{align}
because $J'_{m,n}(\T)$ is an ideal containing  $\sum_{u=1}^my_{u, h_j}.$ From (\ref{lem21}) and  (\ref{per2}) we get $\sum_{u=1}^m (S_{i,j}[u]|T) \in J'_{m,n}(\T).$
\end{proof}
\subsection{Proof of Theorem \ref{skewC}}\label{A4} \begin{proof}We may now prove Theorem \ref{skewC}.  For $Z \in \{\IS, \PT, \T\}$ consider the ideal $J'_{m,n}(Z)$ of  $\Bbbk\langle\textbf{y}_{m \times n}\rangle$. From the filtration (\ref{skewC1}) we obtain the following filtration  \begin{equation}\label{skewC2} 0 \subseteq \frac{ M_{(r)}+J'_{m,n}(Z)}{J'_{m,n}(Z)} \subseteq \frac{ M_{(r-1,1)}+J'_{m,n}(Z)}{J'_{m,n}(Z)} \subseteq \dots \subseteq \frac{ M_{(1^r)}+J'_{m,n}(Z)}{J'_{m,n}(Z)}. \end{equation}

(1) We consider case (1) of Theorem \ref{skewC}, so let $Z=\IS$. We have \begin{equation}
\big(\frac{\Bbbk\langle\textbf{y}_{m \times n}\rangle}{J'_{m,n}(\IS)}\big)_r \cong \frac{ M_{(1^r)}+J'_{m,n}(\IS)}{J'_{m,n}(\IS)}
\end{equation}
as $\IS_m \times \IS_n$-modules.

If $r > \min\{m,n\}$, then by Lemma \ref{dim1skew}(1) we have $\big(\Bbbk\langle\textbf{y}_{m \times n}\rangle/ J'_{m,n}(\IS)\big)_r=0$. Let $r \le  \min\{m,n\}$. Suppose $\la \vdash r$ and $S \in \tab_\lambda([m])$, $T \in \tab_\lambda([n]).$ If either of $S,T$ has 
an entry of weight at least 2, then by Lemma \ref{lemskew1} we have $(S|T) \in J'_{m,n}(\IS)$. Hence the map (\ref{betaskew}) induces a surjective map of $\IS_m \times \IS_n$-modules
\begin{equation}\label{betabarskew} \bar{\beta}_\lambda:\R(m)^{\lambda'} \otimes \R(n)_\lambda \to \frac{ M_{\la}+J'_{m,n}(\IS)}{\dot{M}_{\la} + J'_{m,n}(\IS)}.\end{equation}
By summing with respect to all partitions of $r$ we obtain a surjective map of vector spaces  \begin{equation}\label{surjskew1}\bigoplus_{\lambda \vdash r} \R(m)^{\lambda'} \otimes \R(n)_\lambda \to \frac{ M_{(1^r)}+J'_{m,n}(\IS)}{J'_{m,n}(\IS)}.\end{equation}
Using Lemma \ref{dim1skew}(1) and a computation with dimensions identical to the proof of Theorem \ref{m2} (at the end of part (1)), yields that the domain and codomain of the linear map (\ref{surjskew1}) have equal dimensions. Thus (\ref{surjskew1}) is an isomorphism which implies that the map in  (\ref{betabarskew}) is an isomorphism.

(2) Let $Z=\PT$. The proof of this case is similar to (1).

(3) Let $Z=\T$. This too is similar to (1), but there is one difference. If $r > n$, then by Lemma \ref{dim1skew}(3) we have $\big(\Bbbk\langle\textbf{y}_{m \times n}\rangle/ J'_{m,n}(\T)\big)_r=0$. Let $r \le  n$. Suppose $\la \vdash r$ and $S \in \tab_\lambda([m])$, $T \in \tab_\lambda([n]).$ If $T$ has 
an entry of weight at least 2, then by Lemma \ref{lemskew1}(1) we have $(S|T) \in J'_{m,n}(\PT) \subseteq J'_{m,n}(\T)$. Hence the map (\ref{betaskew}) induces a surjective map of $\mathfrak{S}_m\times \T_n$-modules
\begin{equation}\label{betabarskew3} L_{\lambda}(V_m) \otimes \R(n)_\lambda \to \frac{ M_{\la}+J'_{m,n}(\T)}{\dot{M}_{\la} + J'_{m,n}(\T)}.\end{equation} Consider  the subspace $W_\la$ of $L_{\lambda}(V_m) $ defined in Definition \ref{W}(2).  The kernel of the map (\ref{betabarskew3}) contains the subspace $W_{\lambda} \otimes \R(n)_\lambda$ according to Lemma \ref{lemskew2}. Now by Lemma \ref{presT} we conclude that the map (\ref{betabarskew3}) induces a surjective map  of $\mathfrak{S}_m \times \T_n $-modules
\begin{equation}\label{betabarskew4} L_{\lambda}(U_m) \otimes \R(n)_\lambda \to \frac{ M_{\la}+J'_{m,n}(\T)}{\dot{M}_{\la} + J'_{m,n}(\T)}.\end{equation}
By summing with respect to all partitions of $r$ we obtain a surjective map of vector spaces 
 \begin{equation}\label{surjskew2}\bigoplus_{\lambda \vdash r} L_\la (U_m) \otimes \R(n)_\lambda \to \frac{ M_{(1^r)}+J'_{m,n}(\T)}{J'_{m,n}(\T)}.\end{equation} Finally using Lemma \ref{dim1skew}(3) and a computation with dimensions identical to the proof of  eq. (\ref{c3}) in the proof of Theorem \ref{m2}(3), yields that the domain and codomain of the linear map (\ref{surjskew2}) have equal dimensions. Thus (\ref{surjskew2}) is an isomorphism which implies that the map in  (\ref{betabarskew4}) is an isomorphism. This completes the proof of Theorem \ref{skewC}.\end{proof}
  \section*{Acknowledgments}
We are very grateful to the anonymous referee for detailed and constructive comments and suggestions, which have significantly improved the paper.
	

\begin{thebibliography}{99}
		\bibitem{ABW}Akin K., Buchsbaum D. and Weyman J., Schur functors and Schur complexes, Adv. Math. \textbf{44} (1982), 207--278.
	\bibitem{AnMar} Andre C. A. M. and Martins I. L., Schur-Weyl dualities for the rook monoid: an approach via Schur algebras, Semigr. Forum \textbf{109} (2024), 38--59.
	\bibitem{Ax}Axtell J., Cellularity of generalized Schur algebras via Cauchy decomposition, J. Algebra 572 (2021), 422--460.
\bibitem{Bo}Boffi G., Characteristic-free decomposition of skew Schur functors, J. Algebra \textbf{125} (1989), 288--297.
\bibitem{Cli}Clifford A. H.,  Basic representations of completely simple semigroups. Amer. J. Math. \textbf{82} (1960), 430--434.
\bibitem{DEP}De Concini C., Eisenbud D. and Procesi C., Young diagrams and determinantal varieties, Invent. Math. \textbf{56} (1980),129--165.
	\bibitem{DKR} D\'esarm\'enien J.,  Kung J. and Rota G.-C., Invariant theory, Young bitableaux, and combinatorics, Adv. in Math. \textbf{27} (1978), 63--92.
	\bibitem{Do}Donkin S., \textit{Rational Representations of Algebraic Groups}, LNM \textbf{1140}, Springer, 1985.
	\bibitem{Do2}Donkin S., \textit{The q-Schur Algebra}, Cambridge University Press, 1998.
	\bibitem{Dot}Doty, S., Polynomial representations, algebraic monoids, and Schur algebras of classical type, J. Pure Appl. Algebra  \textbf{123} (1998), 165--199.
	\bibitem{F}Fulton W., \textit{Young Tableaux, With Applications to Representation Theory and Geometry}, vol. \textbf{35}, 
		Cambridge University Press, 1997.
\bibitem{GaPro} Garsia A. M. and Procesi C.,
On certain graded \(S_n\)-modules and the \(q\)-Kostka polynomials,
Adv. Math. \textbf{94 }(1992), 82--138.
\bibitem{GaMaSt}Ganyushkin O., Mazorchuk V. and Steinberg B., On the irreducible representations of a finite semigroup, Proc. AMS \textbf{137} (2009), 3585--3592.
		\bibitem{GaMa}Ganyushkin O. and Mazorchuk V., \textit{Classical Finite Transformation Semigroups, An Introduction}, Algebra and Applications vol. \textbf{9}, Springer, 2009.
		\bibitem{Gr}Green J. A., \textit{Polynomial Representations of} $GL_n$, 2nd edition, LNM \textbf{830}, Springer, 2007.
		\bibitem{Gro}Grood C., A Specht module analog for the rook monoid, Electron. J. Comb. \textbf{9} (2002) $\#$R2.
		\bibitem{JaPe}James D.G. and Peel M.H., Specht series for skew representations of symmetric groups, J. Algebra \textbf{56} (1979), 343--364.
	\bibitem{Kos}Kostant B., Lie group representations on polynomial rings,
	Bull. Amer. Math. Soc. \textbf{69} (1963), no. 4, 518--526.
		\bibitem{Kr}Krause H., Highest weight categories and strict polynomial functors. With an appendix by Cosima Aquilino, in: EMS Ser. Congr. Rep., Representation Theory, Current Trends and Perspectives, Eur. Math. Soc., Zürich (2017),  331--373.
		\bibitem{Kr2}Krause H., \textit{Homological Theory of Representations}, Cambridge University Press, 2021.
		\bibitem{Ma}Maliakas M., Cauchy decompositions and invariants, Math. Z.  \textbf{235} (2000), 629--650.
	 \bibitem{MaSr}Mazorchuk V. and Srivastava S, Jucys–Murphy elements and Grothendieck groups for generalized rook monoids, J. Comb. Algebra \textbf{6} (2022), 185--222.
	  \bibitem{Mun1}Munn W., On semigroup algebras. Proc. Cambridge Philos. Soc. \textbf{51} (1955), 1--15.
	  \bibitem{Mun2}Munn W., Matrix representations of semigroups. Proc. Cambridge Philos. Soc. \textbf{53} (1957), 5--12.
	 \bibitem{Mun}Munn W., The characters of the symmetric inverse  semigroups, Proc. Cambridge  Phil. Soc. \textbf{53} (1957), 13--18.
	 \bibitem{OZ}Orellana, R. and Zabrocki, M., The Hopf structure of symmetric group characters as symmetric functions, Algebr. Comb. \textbf{4} (2021), 551--574.
	 \bibitem{Poz}Ponizovskii I., On matrix representations of associative systems. Mat. Sb. N.S. \textbf{38} (1956), 241--260.
	 \bibitem{Pu}Putcha M. S., Complex representations of finite monoids, Proc. London Math. Soc. 73 (1996), 623--641.
	 \bibitem{RR}Reiner V. and Rhoades B., Harmonics and graded Ehrhart theory, arXiv:2407.06511v3.
	 \bibitem{Rho}Rhoades B., Increasing subsequences, matrix loci and Viennot shadows, Forum Math. Sigma \textbf{12} (2024): e97.
	\bibitem{Sag}Sagan B. E., \textit{The Symmetric Group. Representations, Combinatorial Algorithms, and Symmetric Functions}, GTM vol. \textbf{203}, 2nd edn. Springer, 2001.
	\bibitem{St}Stanley R. P., \textit{Enumerative Combinatorics}, Volume 2, 2nd edn. Cambridge University Press, 2024. 
	\bibitem{Ste}Steinberg B., \textit{Representation Theory of Finite Monoids}, Universitext, Springer, 2016.
	\bibitem{Ste2}Steinberg B., The global dimension of the full transformation monoid (with an Appendix by V. Mazorchuk and B. Steinberg), Algebr. Represent. Theory \textbf{19} (2016), 731--747.
	\bibitem{Sol}Solomon L., Representations of the rook monoid, J. Algebra \textbf{256} (2002), 309--342.
		\bibitem{W} Weyman J., \textit{Cohomology of vector bundles and syzygies},  Cambridge University Press, vol. \textbf{149}, Cambridge, 2003.
	\end{thebibliography}
	\end{document}